\documentclass[11pt]{amsart}
\usepackage{amscd,amssymb,float}
\input xy
\xyoption{all}

\newtheorem{theorem}{Theorem}[section]
\newtheorem{lemma}[theorem]{Lemma}

\newtheorem{corollary}[theorem]{Corollary}
\newtheorem{proposition}[theorem]{Proposition}

\newtheorem{conjecture}[theorem]{Conjecture}

\renewcommand{\leq}{\leqslant}
\renewcommand{\geq}{\geqslant}

\theoremstyle{definition}

\theoremstyle{definition}
\newtheorem{remark}[theorem]{Remark}
\numberwithin{equation}{section}

\newcommand{\ve}{\varepsilon}
\newcommand{\vp}{\varphi}
\newcommand{\mn}{\sqrt{-1}}
\newcommand{\ov}[1]{\overline{#1}}
\newcommand{\de}{\partial}

\newcommand{\ddbar}{\sqrt{-1} \partial \overline{\partial}}
\newcommand{\ti}[1]{\tilde{#1}}

\renewcommand{\hom}{\mathrm{Hom}}

\numberwithin{equation}{section} \numberwithin{figure}{section}

\author{Valentino Tosatti}
\address{Department of Mathematics, Northwestern University, 2033 Sheridan Road, Evanston, IL 60208}
\email{tosatti@math.northwestern.edu}
\author{Yuguang Zhang}
\address{Yau Mathematical Sciences Center,  Tsinghua University,  Beijing 100084, P.R.China.}
\address{Current address: Department of Mathematical Sciences, University of Bath, Bath, BA2 7AY, UK.}
\email{yuguangzhang76@yahoo.com}
\title{Collapsing Hyperk\"ahler manifolds}

\begin{document}

\begin{abstract}
Given a projective hyperk\"ahler manifold with a holomorphic Lagrangian fibration, we prove that hyperk\"ahler metrics with volume of the torus fibers shrinking to zero collapse in the Gromov-Hausdorff sense (and smoothly away from the singular fibers) to a compact metric space which is a half-dimensional special K\"ahler manifold outside a singular set of real Hausdorff codimension $2$, and is homeomorphic to the base projective space.
\end{abstract}

\maketitle

\section{Introduction}
Let $M^m$ be a compact Calabi-Yau manifold, which for us is a compact K\"ahler manifold $M^m$ with $c_1(M)=0$ in $H^2(M,\mathbb{R})$. Yau's Theorem \cite{Ya} shows that given any K\"ahler class $[\alpha]$ on $M$ we can find a unique representative $\omega$ of $[\alpha]$ which is a {\em Ricci-flat K\"ahler metric}. The basic problem that we study in this paper is to understand the limiting behavior of such Ricci-flat metrics if we degenerate the class $[\alpha]$. More precisely, we fix a class $[\alpha_0]$ on the boundary of the K\"ahler cone and for $0<t\leq 1$ we let $\ti{\omega}_t$ be the unique Ricci-flat K\"ahler metric in the class $[\alpha_0]+t[\omega_M]$, where $\omega_M$ is a fixed Ricci-flat K\"ahler metric on $M$, and we wish to understand the behavior of $(M,\ti{\omega}_t)$ as $t\to 0$. The metrics $\ti{\omega}_t$ satisfy the equation
\begin{equation}\label{eq1}  \tilde{\omega}_{t}^{m}=c_{t} t^{m-n}\omega_M^m,\end{equation}
for some explicit constants $c_t$ which approach a positive constant as $t\to 0$. Up to scaling the whole setup, we may assume without loss of generality that $c_t \rightarrow 1$ when $t\rightarrow 0$.

This question has been extensively studied in the literature, the most relevant works being \cite{GW,To0,To1,GTZ,GTZ2,HT,TWY,TZ2,CT,So}, see also the surveys \cite{To2,To3,Zh}. In particular, decisive results in the {\em non-collapsing} case when $\int_M\alpha_0^n>0$ have been obtained in \cite{To0,CT,So}. In this paper we consider the more challenging {\em collapsing} case when $\int_M\alpha_0^n=0$, and we will always assume that $[\alpha_0]=f^*[\omega_N]$ where $(N^n,\omega_N)$ is a compact K\"ahler manifold with $0<n<m$ and $f:M\to N$ is  a holomorphic surjective map with connected fibers (i.e. a fiber space). This is the same setup as in \cite{To1,GTZ,GTZ2,HT,TWY,TZ2}, and as explained there, in this case there are proper analytic subvarieties $S'\subset N$ and $S=f^{-1}(S')\subset M$ such that $f:M\backslash S\to N_0:=N\backslash S'$ is a proper submersion with fibers $M_{y}=f^{-1}(y), y\in N\backslash S',$ smooth Calabi-Yau $(m-n)$-folds.

In \cite{TWY}, building upon the earlier \cite{To1}, it is shown that there is a K\"ahler metric $\omega$ on $N_{0}$ such that as $t\to 0$ the metrics $\tilde{\omega}_{t}$  converges  to $f^{*}\omega$ uniformly on compact subsets  of $M\backslash S$, and $\omega$ satisfies
\begin{equation}\label{wp}
{\rm Ric}(\omega) =\omega_{WP}\geq 0,
\end{equation}
on $N_{0}$,  where $\omega_{WP}$ is a {\em Weil-Petersson form} which measures the variation of the complex structures of the fibers $M_y$ (see e.g. \cite{To1,ST}). This is improved to smooth convergence on compact subsets of $M\backslash S$ in \cite{GTZ,HT,TZ2} when the fibers $M_y$ are tori (or finite \'etale quotients of tori). Explicit estimates are also obtained for $\ti{\omega}_t$ near $S$, but these blow up very fast near $S$.

Our main concern is understanding the possible collapsed Gromov-Hausdorff limits of $(M,\ti{\omega}_t)$ as $t\to 0$, and their singularities. In this regard, we have the following conjecture (see \cite[Question 4.4]{To2} \cite[Question 6]{To3}), which is motivated by an analogous conjecture by Gross-Wilson \cite{GW}, Kontsevich-Soibelman \cite{KS,KS2} and Todorov \cite{Man} for collapsed limits of Ricci-flat K\"ahler metrics on Calabi-Yau manifolds near a large complex structure limit:
 \begin{conjecture}\label{con1}
 If   $(X, d_{X})$  denotes the  metric completion of $(N_{0}, \omega)$, and $S_{X}=X\backslash N_{0}$, then
\begin{itemize}
\item[(a)] $(X, d_{X})$ is a compact length metric space and $S_X$ has real Hausdorff codimension at least $2$.
\item[(b)] We have that $$(M, \tilde{\omega}_{t}) \stackrel{d_{GH}}\longrightarrow  (X, d_{X}),$$ when $t \to 0.$
\item[(c)] $X$ is homeomorphic to $N$.
\end{itemize}
\end{conjecture}

This conjecture was proved by Gross-Wilson \cite{GW} when $f:M\to N$ is an elliptic fibration of $K3$ surfaces with only $I_1$ singular fibers. In our earlier work with Gross \cite{GTZ2} we proved Conjecture \ref{con1} completely in the case when $\dim N=1$. Very recently, conditionally to a certain H\"older estimate for solutions of a family of Monge-Amp\`ere equations, a new proof was obtained in \cite{Li} when $\dim M=3,\dim N=1,$ the generic fibers $M_y$ are $K3$ surfaces and the singular fibers are nodal $K3$ surfaces, which also gives better estimates near and on the singular fibers.

For bases $N$ of general dimension $n$, the only known partial result towards Conjecture \ref{con1} is the one proved in \cite{TWY,GTZ}: if $(X,d_X)$ is the Gromov-Hausdorff limit of a sequence $(M,\ti{\omega}_{t_i}), t_i\to 0$ (such limits always exist up to passing to subsequences), then there is a homeomorphism $\psi:N_0\to X_0$ onto a dense open subset $X_0\subset X$ such that $\psi:(N_0,\omega)\to (X_0,d_X|_{X_0})$ is a local isometry.

Our main result is the following:

\begin{theorem}\label{theorem1}
Conjecture \ref{con1} holds when $M$ is a projective hyperk\"ahler manifold.
\end{theorem}
As proved in \cite{GTZ2}, in this case the limiting metric $\omega$ on $N_0$ is a {\em special K\"ahler metric} in the sense of \cite{Fr}. In this case, the base $N$ is always $\mathbb{CP}^n$ \cite{Hw} and the fibers $M_y,y\in N_0,$ are holomorphic Lagrangian $n$-tori \cite{Ma,Ma2}, so that $f$ is an {\em algebraic completely integrable system} over $N_0$. A classical result of Donagi-Witten \cite{DW} (see also \cite{Fr, Hi}) shows that the base of an algebraic completely integrable system admits a special K\"ahler metric, and our result shows that this metric arises as the collapsed limit of hyperk\"ahler metrics on the total space.

An application of our result is the revised Strominger-Yau-Zaslow (SYZ) conjecture due to  Gross-Wilson \cite{GW}, Kontsevich-Soibelman \cite{KS,KS2} and Todorov \cite{Man} (note that the statement of the conjecture in \cite{KS,KS2} also covers the hyperk\"ahler case).
 As explained in \cite{GTZ}, Theorem \ref{theorem1} implies a positive solution to such  conjecture for collapsed limits of hyperk\"ahler metrics near large complex structure limits which arise via hyperk\"ahler rotation from our setting above:

\begin{corollary}\label{ksthm}
The conjecture of Gross-Wilson \cite[Conjecture 6.2]{GW}, Kontsevich-Soibelman \cite[Conjectures 1 and 2]{KS} and Todorov \cite[p. 66]{Man} holds for those large complex structure families of hyperk\"ahler manifolds which arise from the setup of Theorem \ref{theorem1} via hyperk\"ahler rotation as in \cite[Theorem 1.3]{GTZ}.
\end{corollary}

Indeed, this follows exactly as in \cite[Theorem 1.3]{GTZ}, using Theorem \ref{theorem1} together with our earlier results in \cite[Theorem 1.2]{GTZ2}. The key new information provided by Theorem \ref{theorem1}, which was not available in \cite{GTZ,GTZ2}, is the uniqueness of the Gromov-Hausdorff limit, which is identified with the metric completion of the smooth part and is homeomorphic to the base, and the fact that it has singularities in real codimension at least $2$. This completes the program we started in \cite{GTZ} to extend Gross-Wilson's theorem on large complex structure limits of $K3$ surfaces \cite{GW} to higher-dimensional hyperk\"ahler manifolds.  Furthermore, by combining it with \cite{GTZ2},  Theorem \ref{theorem1} gives more  precise information about the limit space, which was predicted by \cite{GW,KS,KS2}.  This is explained in detail in section \ref{sectsyz} (see Theorem \ref{syzthm} there) in a slightly  more general  setup  than \cite{GTZ}.

We now give a brief outline of the paper. In section \ref{sectmain} we extend and sharpen a method introduced in our earlier work \cite{GTZ2} (when $\dim N=1$) and show that to prove parts (a) and (b) of Conjecture \ref{con1} in general it suffices to obtain an upper bound for the limiting metric $\omega$ near $S'$ (which may be assumed to be a simple normal crossings divisor after a modification) in terms of an {\em orbifold K\"ahler metric} up to a logarithmic factor. In section \ref{secthk} we give some improvements of earlier results of ours, and state the more precise estimate that we obtain in the hyperk\"ahler case, which implies the estimate needed in section \ref{sectmain}. We also show how the more precise estimate implies part (c) of Conjecture \ref{con1}. This estimate is then proved in section \ref{sectspec} by showing that the coefficients of the special K\"ahler metric $\omega$ are essentially given by periods of the Abelian varieties which are the fibers of $f$, and the blowup rate of these periods can be controlled using degenerations of Hodge structures. Lastly, in section \ref{sectsyz} we explain how Theorem \ref{theorem1} fits into the SYZ picture of mirror symmetry for hyperk\"ahler manifolds.
\\

{\bf Acknowledgments. }We thank H.-J. Hein, M. Popa, J. Song and G. Sz\'ekelyhidi for useful discussions, and Y.S. Zhang and the referees for comments. Part of this work was done during the first-named author's visits to the Yau Mathematical Sciences Center at Tsinghua University in Beijing and to the Institut Henri Poincar\'e in Paris (supported by a Chaire Poincar\'e at IHP funded by the Clay Mathematics Institute) and during the second-named author's visit to the Department of Mathematics at Northwestern University, which we   would like to thank for the hospitality. The first-named author was partially supported by NSF grant DMS-1610278.

\section{Gromov-Hausdorff Collapsing}\label{sectmain}
In this section we reduce parts (a) and (b) of Conjecture \ref{con1} in general to proving a suitable bound for the limiting metric near the discriminant locus of the map $f$, in terms of an orbifold K\"ahler metric on some log resolution of the discriminant locus of $f$. A stronger bound will then be proved in section \ref{sectspec} for hyperk\"ahler manifolds, which in section \ref{secthk} will be shown to be also sufficient to prove part (c) of the Conjecture. We decided to follow this presentation in order to highlight precisely what estimates are needed for each part of the Conjecture, and also because we envision that the estimate in Theorem \ref{main} may be more approachable in the general (not necessarily hyperk\"ahler) case than the stronger estimate in Theorem \ref{prop1}.

Let $f:M^m\to N^n$ be as in the Introduction, so $M$ is a compact Calabi-Yau manifold, $N$ is a compact K\"ahler manifold, $f$ is holomorphic surjective with connected fibers, and $0<n<m$. The discriminant locus of $f$ (i.e. the locus of critical values of $f$) is denoted by $S'\subset N$, and is a proper analytic subvariety of $N$. As an aside, we will see in Theorem \ref{triv} below (a small extension of our earlier result in \cite{TZ}) that $S'=\emptyset$ happens if and only if $f$ is a holomorphic fiber bundle, with base $N$ and fiber also Calabi-Yau manifolds. This is a very special situation, and in fact never happens if $M$ is hyperk\"ahler (since in this case $N\cong\mathbb{CP}^n$ \cite{Hw}). In any case, we may assume in the following that $S'\neq\emptyset$, since otherwise Conjecture \ref{con1} follows immediately from the results in \cite{TWY} which give uniform convergence of $\ti{\omega}_t$ to $f^*\omega$.

Let $\pi: \ti{N} \to N$ be a modification such that $E=\pi^{-1}(S')$ is a divisor with simple normal crossings and $\pi: \pi^{-1}(N_{0}) \to N_{0}$ is biholomorphic, where $N_0=N\backslash S'$. Write $E=\bigcup_{j=1}^\mu E_j$ for the decomposition of $E$ in irreducible components, so each $E_j$ is a smooth irreducible divisor and the $E_j$'s intersect in normal crossings. There is an integer $0\leq \ell\leq \mu$ so that the divisors $E_j$ with $1\leq j\leq \ell$ are $\pi$-exceptional, while $E_j$ with $\ell+1\leq j\leq \mu$ are proper transforms of divisors in $S'$. The limit cases $\ell=0,\mu$ are allowed, where $\ell=0$ means that $S'$ is already a simple normal crossings divisor and $\pi=\mathrm{Id}$, while $\ell=\mu$ means that $S'$ is of (complex) codimension at least $2$ in $N$.

Given natural numbers $m_i\in \mathbb{N}_{>0}$, $1\leq i\leq \mu$, there is a well-defined notion of {\em orbifold K\"ahler metric} $\omega_{\rm orb}$ on $\ti{N}$ with singularities along $E$ with orbifold order $m_i$ along each component $E_i$. Any such metric is a smooth K\"ahler metric on $\ti{N}\backslash E$ such that in any local chart $U$ (a unit polydisc with coordinates $(w_1,\dots, w_n)$) centered at a point of $E$ adapted to the normal crossings structure (so $E\cap U$ is given by $w_1\cdots w_k=0$ for some $1\leq k\leq n$, and say that $\{w_i=0\}=E_{j_i}\cap U$ for some $1\leq j_i\leq \mu$ and all $1\leq i\leq k$) we have that pulling back $\omega_{\rm orb}$ by the local uniformizing map $q:\ti{U}\to U$ ($\ti{U}$ is also the unit polydisc in $\mathbb{C}^n$) given by
$$q(w_1,\dots,w_n)=(w_1^{m_{j_1}},\dots,w_k^{m_{j_k}},w_{k+1},\dots,w_n),$$
the resulting metric on $\ti{U}\backslash\{w_1\cdots w_k=0\}$ extends smoothly to a K\"ahler metric on $\ti{U}$. This implies that on $U\backslash\{w_1\cdots w_k=0\}$ the metric $\omega_{\rm orb}$ is uniformly equivalent to the model
$$\sum_{i=1}^k \frac{\mn dw_i\wedge d\ov{w}_i}{|w_i|^{2(1-1/m_{j_i})}}+\sum_{i=k+1}^n \mn dw_i\wedge d\ov{w}_i.$$
We also fix a defining section $s_i$ of the divisor $E_i$ and a smooth Hermitian metric $h_i$ on $\mathcal{O}(E_i)$, for all $1\leq i\leq \mu$. Then a short calculation shows that given any K\"ahler metric $\omega_{\ti{N}}$ on $\ti{N}$ and any $\ve>0$ sufficiently small, the formula
$$\omega_{\rm orb}=\omega_{\ti{N}}+\ve\sum_{i=1}^\mu\ddbar|s_i|^{\frac{2}{m_i}}_{h_i},$$
defines an orbifold K\"ahler metric on $\ti{N}$ with orbifold order $m_i$ along each $E_i$. This shows that we can always find orbifold K\"ahler metrics adapted to any given orbifold structure.

Recall that on $N_0$ we have the limiting K\"ahler metric $\omega$, which is constructed in \cite{ST,To1} by solving a suitable Monge-Amp\`ere equation, and which satisfies \eqref{wp}.

The following result can be viewed as a generalization of \cite[Section 3]{GTZ2} to higher dimensions:
\begin{theorem}\label{main}
Suppose that there is a constant $C>0$ and natural numbers $d\in\mathbb{N}$ and $m_i\in \mathbb{N}_{>0}$, $1\leq i\leq \mu$, such that on $\pi^{-1}(N_0)$ we have
\begin{equation}\label{to0}
\pi^*\omega\leq C\left(1-\sum_{i=1}^\mu\log|s_i|_{h_i}\right)^d \omega_{\rm orb},
\end{equation}
where $\omega_{\rm orb}$ is an orbifold metric with orbifold order $m_i$ along each component $E_i$.
Then parts (a) and (b) of Conjecture \ref{con1} hold.
\end{theorem}
\begin{proof}
Let us define $E'_{n+1}=\emptyset$ and for $1\leq p\leq n$ define recursively
\begin{equation}\label{strata}
E'_p=\bigcup_{|J|=p} (E_{j_1}\cap\cdots\cap E_{j_p})\backslash E'_{p+1},
\end{equation}
where the union is over all multiindices $J=(j_1,\dots,j_p)$ with $1\leq j_1,\dots,j_p\leq \mu$. Therefore each $E'_p$ is a (possibly empty) disjoint union of smooth connected $(n-p)$-dimensional relatively compact complex submanifolds of $\ti{N}$ (the real $(2n-2p)$-Hausdorff measure of $E'_p$ is therefore finite), and we can write $E=\bigcup_{p=1}^n E'_p.$ Note also that if $U$ is any small open neighborhood of $E'_{p+1}$, then $E'_{p}\backslash U$ is compact. Using this we see that for every small $\rho,\beta>0$ we can find a covering of $E$ by $N(\rho)$ open sets $\{V_i(\rho)\}\subset\ti{N}$ such that we have
\begin{equation}\label{limit}
\rho^{2n-2+\beta}N(\rho)\to 0,
\end{equation}
as $\rho\to 0$, and each $V_i(\rho)$ is contained in a chart with coordinates $(w_1,\dots,w_n)$ defined in the unit polydisc $\Delta^n$ where $E$ is given locally by $w_1\cdots w_k=0$ for some $1\leq k\leq n$, and in this chart we have
$$V_i(\rho)=\{w\in \Delta^n\ |\ |w_j|< \rho^{m_j}, \textrm{ for }1\leq j\leq k, \textrm{ and }|w_j|<\rho, \textrm{ for }k+1\leq j\leq n\},$$
where for simplicity of notation we will write $m_j$ for the orbifold order along $\{w_j=0\}$.
Define a map $q:\Delta^n\to\Delta^n$ by
$$q(w_1,\dots,w_n)=(w_1^{m_1},\dots,w_k^{m_k},w_{k+1},\dots,w_n),$$
so that
$q^{-1}(V_i(\rho))$ equals the polydisc of radius $\rho$, which we will denote by $\Delta^n(\rho)$. Then our assumption \eqref{to0} implies that on $\Delta^n\backslash E$ we have that
\begin{equation}\label{estimatio}
q^*\pi^*\omega\leq C \left(1-\sum_{i=1}^k\log |w_i|\right)^d \omega_E,
\end{equation}
for some $C>0, d\in\mathbb{N}$, where $\omega_E$ is the Euclidean metric on $\Delta^n$.
Using this, we claim that for any two points $q_1,q_2\in \Delta^{*,k}(\rho)\times\Delta^{n-k}(\rho)\subset\Delta^n(\rho)$ there is a path $\ti{\gamma}$ connecting them such that $\ti{\gamma}\subset \Delta^{*,k}(\rho)\times\Delta^{n-k}(\rho)$ and
\begin{equation}\label{too0}
{\rm length}_{q^*\pi^*\omega} (\ti{\gamma}) \leqslant C \rho (-\log \rho)^{d}.
\end{equation}
Indeed, if we denote $w_{i}=r_{i}e^{\sqrt{-1} \theta_{i}}$, then on $\Delta^n\backslash E$ the estimate \eqref{estimatio} translates to
\begin{equation}\label{estima}
q^*\pi^*\omega\leqslant C \left(1-\sum_{i=1}^k\log r_{i}\right)^d \sum_{j=1}^n (dr_{j}^{2}+r_{j}^{2}d\theta_{j}^{2}).
\end{equation}
Write
$$q_1=(r_1e^{\mn\theta_1},\dots,r_ne^{\mn\theta_n}),$$
where $0<r_j<\rho$ for $1\leq j\leq k,$ while $0\leq r_j<\rho$ for $k+1\leq j\leq n$, and if $r_j=0$ for any such $j$ then we set $\theta_j=0$.
We then define a path
$$\gamma_1(s)= ((s\rho/2+(1-s)r_1) e^{\sqrt{-1}\theta_1},\dots,(s\rho/2+(1-s)r_n) e^{\sqrt{-1}\theta_n}),$$ with $0\leq s\leq 1$ gives a path in $\Delta^{*,k}(\rho)\times\Delta^{n-k}(\rho)$ whose initial point is $q_1$ and whose endpoint lies on the distinguished boundary of $\Delta^n(\rho/2)$, given by
$$S(\rho/2)=\{w \in \Delta^{n}(\rho)\ |\quad |w_j|=\rho/2, 1\leq j\leq n\}.$$
The Euclidean norm of $\gamma_1'$ is at most $\rho$, and so using \eqref{estima} we obtain
\[\begin{split}
{\rm length}_{q^*\pi^*\omega}  (\gamma_{1})  &\leqslant  C \rho\int_{0}^{1} \left(1-\sum_{i=1}^k\log (s\rho/2+(1-s)r_i)\right)^{\frac{d}{2}} ds\\
 &\leq C\rho\int_0^1  \left(1-k\log (s\rho/2)\right)^{\frac{d}{2}} ds\\
 &\leq C\rho(-\log \rho)^{d},
\end{split}\]
where we used the fact that $\rho$ is small to increase the power of $-\log$.
On the other hand the distinguished boundary $S(\rho/2)$ is diffeomorphic to the real torus $T^n$ (in particular it is connected), and using again \eqref{estima}
we see that
$$ {\rm diam}_{q^*\pi^*\omega}  (S(\rho/2))  \leqslant  C \rho(-\log \rho)^{\frac{d}{2}}\leqslant   C \rho(-\log \rho)^{d}. $$
Therefore we conclude that $q_1$ and $q_2$ can indeed by joined by a curve $\ti{\gamma}\subset \Delta^{*,k}(\rho)\times\Delta^{n-k}(\rho)$ satisfying \eqref{too0}.
Considering the image $\gamma=q(\ti{\gamma})$, whose $\pi^*\omega$-length is equal to the $q^*\pi^*\omega$-length of $\ti{\gamma}$, we conclude that every two points in $V_i(\rho)\backslash E$ can be joined by a path $\gamma$ contained in $V_i(\rho)\backslash E$ with
\begin{equation}\label{too1}
{\rm length}_{\pi^*\omega} (\gamma) \leqslant C \rho (-\log \rho)^{d}.
\end{equation}
Since the open sets $\{V_i(\rho)\}$ cover $E$, and since $\rho(-\log \rho)^{d}\to 0$ as $\rho\to 0$, it follows in particular that $$\sup_{y_{1},y_{2}\in N_{0}} d_{\omega}(y_{1},y_{2}) \leqslant C,$$ where
$d_{\omega}$ is the metric space structure   on $N_{0}$ induced by $\omega$. This implies that the metric completion of $(N_0,d_\omega)$ is compact.

Now pick any sequence $t_k\to 0$ such that $(M,\ti{\omega}_{t_k})$ converges in the Gromov-Hausdorff topology to a compact length metric space $(X,d_X)$. As we recalled in the Introduction, in \cite[Corollary 1.4]{TWY} (see also \cite[Theorem 1.2]{GTZ}) we constructed a local isometric embedding
of $(N_0,\omega)$ into $(X,d_X)$ with open dense image $X_0\subset X$ via a homeomorphism $\psi:N_0\to X_0$. Call $S_X=X\backslash X_0$.
The density of $X_0$ implies that for every fixed $\rho>0$ the set
$$\bigcup_{i=1}^{N(\rho)}\ov{\psi(\pi(V_i(\rho))\cap N_0)},$$
covers $S_X$.
 Then the fact that $(X,d_X)$ is a length space implies that for every $i$ we have
\[\begin{split}{\rm diam}_{d_{X}} (\overline{\psi(\pi(V_i(\rho))\cap N_0)}) &={\rm diam}_{d_{X}} (\psi(\pi(V_i(\rho))\cap N_0))\\
&=\sup_{p,q\in \psi(\pi(V_i(\rho))\cap N_0)}\inf_{\eta}{\rm length}_{d_X}(\eta),
\end{split}\]
where the infimum is over all curves $\eta$ in $X$ joining $p$ and $q$. But we have just shown that $p$ and $q$ can be joined by curves of the form $\psi(\pi(\gamma))$, with $\gamma\subset V_i(\rho)\backslash E$ satisfying \eqref{too1},
and since $\psi$ is a local isometry we have that
\begin{equation}\label{lengthp}
{\rm length}_{d_X}(\psi(\pi(\gamma)))={\rm length}_{\omega}(\pi(\gamma))={\rm length}_{\pi^*\omega}(\gamma),
\end{equation}
for any such curve $\gamma$.
We then conclude that
$${\rm diam}_{d_{X}} (\overline{\psi(\pi(V_i(\rho))\cap N_0)}) \leqslant C  \rho(-\log \rho)^{d},$$
for all $\rho>0$ small and for $C>0$ independent of $\rho$.

Fix now small $\beta,\ve>0$, and given any $\eta>0$, choose $\rho>0$ small so that $ C  \rho(-\log \rho)^{d}<\eta.$ We estimate
\begin{eqnarray*}
\mathcal{H}_{d_{X}, \eta}^{2n-2+\beta+\ve}(S_{X}) & \leqslant   & \sum_{i=1}^{N(\rho)}\varpi_{2n} {\rm diam}_{d_{X}}^{2n-2+\beta+\ve}  (\overline{\psi(\pi(V_{i}(\rho))\cap N_{0})})\\  & \leqslant &  C N(\rho)\rho^{2n-2+\beta+\ve}(-\log \rho)^{d(2n-2+\beta+\ve)}\\ & =& C \rho^{\ve}(-\log \rho)^{d(2n-2+\beta+\ve)}  \left(N(\rho)\rho^{2n-2+\beta}\right)\\ & \to & 0,
\end{eqnarray*}
as $\rho\to 0$ thanks to \eqref{limit},
where $\varpi_{2n}$ denotes the volume of unit ball in $\mathbb{R}^{2n}$. Note that as $\eta\to 0$ then $\rho\to 0$ as well.
Thus
\begin{equation}\label{measure}
\mathcal{H}_{d_{X}}^{2n-2+\beta+\ve}(S_{X})=\lim_{\eta \to 0} \mathcal{H}_{d_{X}, \eta}^{2n-2+\beta+\ve}(S_{X})=0,
\end{equation}
for any small $\beta,\ve>0$, and so we conclude that $\dim_{\mathcal{H} }S_{X}\leqslant 2n-2$. This proves part (a) of Conjecture \ref{con1}.

Note also that for any two points $x,y\in N_0$ we have
\begin{equation}\label{length}
d_X(\psi(x),\psi(y))\leq d_\omega(x,y).
\end{equation}
Indeed, for any $\ve>0$ there is a path $\gamma_\ve$ in $N_0$ joining $x$ and $y$ with
$\mathrm{length}_{\omega}(\gamma_\ve)\leq d_\omega(x,y)+\ve$. From \eqref{lengthp} we see that
$$\mathrm{length}_{\omega}(\gamma_\ve)=\mathrm{length}_{d_X}(\psi(\gamma_\ve))\geq d_X(\psi(x),\psi(y)),$$ and letting $\ve\to 0$ proves \eqref{length}.

If we let $\omega_M$ be any Ricci-flat K\"ahler metric on $M$, and let $\nu$ be the reduced measure constructed in  \cite[Section 5]{GTZ}, then $\nu(S_{X})=0$ \cite[Remark 5.3]{GTZ} and \cite[Section 5]{GTZ} shows that there exist constants $\upsilon,c>0$ such that
for any $K\subset N_{0}$,
$$\nu(K)=\upsilon \int_{f^{-1}(K)} \omega_M^m=\upsilon \int_{K} f_{*}(\omega_M^m)=c\upsilon\int_{K}\omega^{n}$$ because, as explained for example in Section 4 in \cite{To1}, on $N_0$ the metric $\omega$ satisfies
\begin{equation}\label{eq3} \omega^{n}=cf_{*}(\omega_M^m), \end{equation}
for some explicit constant $c>0$.  Therefore
$$\nu(K)=\lambda {\rm Vol}_{\omega}(K)=\lambda\mathcal{H}^{2n}_{d_{X}}(K),$$ for some constant $\lambda>0$.
Thanks to  \cite[Theorem 1.10]{CC1} we have that $\nu$ is a Radon measure, and then the same argument as in \cite[p.105]{GTZ2} shows that $\nu= \lambda \mathcal{H}^{2n}_{d_{X}}$ as measures on $X$. If we let $\nu_{-1}$ be the measure induced by $\nu$ ``in codimension $1$'', as defined in  \cite[Section 2]{CC2} (see also the discussion in \cite[p.106]{GTZ2}), then we deduce that $$\nu_{-1}(S_{X})=\lambda'  \mathcal{H}^{2n-1}_{d_{X}}(S_{X})=0,$$
using \eqref{measure}, for some positive constant $\lambda'$.
We then apply \cite[Theorem 3.7]{CC2} which shows that given any $x_1\in X_0$ for $\mathcal{H}^{2n}_{d_X}$-almost all $y\in X_0$ there exists a minimal geodesic from $x_1$ to $y$ which lies entirely in $X_0$. In particular, given any two points $x_1,x_2\in X_0$ and $\delta>0$, there is a point $y\in X_0$ with $d_X(x_2,y)<\delta$ which can be joined to $x_1$ by a minimal geodesic $\eta_1$ contained in $X_0$. Furthermore we can take $y$ close enough to $x_2$ so that it can also be joined to $x_2$ by a curve $\eta_2$ contained in $X_0$ with $d_X$-length at most $\delta$. Concatenating $\eta_1$ and $\eta_2$ we obtain a curve $\eta$ in $X_0$ joining $x_1$ to $x_2$ with
$${\rm length}_{d_X}(\eta)\leqslant d_{X} (x_{1}, y)+ \delta\leq d_X(x_1,x_2)+2\delta.$$
Since $\psi:N_0\to X_0$ is a homeomorphism, we conclude that given any two points $q_1,q_2\in N_0$ and $\delta>0$, there is a curve $\gamma$ in $X_0$ joining $q_1$ and $q_2$ with
$${\rm length}_\omega(\gamma)\leq d_{X} (\psi(q_{1}), \psi(q_{2}))+ 2\delta.$$
Therefore, thanks to \eqref{length}, we conclude that
$$d_{X} (\psi(q_{1}), \psi(q_{2})) \leqslant d_{\omega} (q_{1}, q_{2})\leqslant{\rm length}_{\omega} (\gamma) \leqslant d_{X} (\psi(q_{1}), \psi(q_{2}))+ 2\delta.$$
Letting $\delta \to 0$, we conclude that
$$d_{\omega} (q_{1}, q_{2}) =  d_{X} (\psi(q_{1}), \psi(q_{2})).$$
Hence $\psi:(N_0,\omega)\to (X_0,d_X)$ is a global isometry, and since $X_0$ is dense in $X$ this implies that $(X, d_{X})$ is isometric to the metric completion of $(N_{0}, \omega)$.
This proves part (b) of Conjecture \ref{con1}.
\end{proof}

The method that we developed to prove Theorem \ref{main} is quite robust, and it applies to other setups as well, see e.g. \cite{ZS} for a very recent work that uses our result in different settings.

\section{Metrics on torus fibrations}\label{secthk}
In this section we prove some general results about metrics on torus fibrations, extending our earlier work in \cite{GTZ,TZ}, and in Theorem \ref{prop1} we state the main estimate which holds in the hyperk\"ahler case, which implies estimate \eqref{to0}, and which will be proved in section \ref{sectspec}. We also show that the main estimate implies part (c) of Conjecture \ref{con1}.

\subsection{Semi-flat forms on torus fibrations}
We start with a general discussion. Let $(M_0^m,\omega_M)$ be a possibly noncompact K\"ahler manifold with a proper holomorphic submersion $f:M_0^m\to N_0^n$ with connected fibers onto a complex manifold $N_0$, $0<n<m$. Assume that all the fibers $M_y=f^{-1}(y), y\in N_0$ are complex tori, so that $M_y\cong\mathbb{C}^{m-n}/\Lambda_y$, for some lattice $\Lambda_y\subset \mathbb{C}^{m-n}$, and that $f$ admits a holomorphic section $\sigma:N_0\to M_0$.

\begin{theorem}\label{semiflat}
Under these assumptions, there is a unique closed semipositive $(1,1)$ form $\omega_{\rm SF}$ on $M_0$, such that $\omega_{\rm SF}|_{M_y}$ is the unique flat K\"ahler metric cohomologous to $\omega_M|_{M_y}$, and such that given any coordinate ball $B\subset N_0$, and any trivialization $f^{-1}(B)\cong  (B\times\mathbb{C}^{m-n})/\Lambda$ which maps $\sigma$ to the zero section, the form $\omega_{\rm SF}$ is given by the explicit formula of \cite{GTZ,He,HT}.
\end{theorem}
The last point means the following: the universal cover of $f^{-1}(B)$ is in fact biholomorphic to $p:B\times\mathbb{C}^{m-n}\to (B\times\mathbb{C}^{m-n})/\Lambda\cong f^{-1}(B)$ (see \cite{GTZ}), we may assume that $p$ pulls back $\sigma$ to the zero section,
and we can then write $p^*\omega_{\rm SF}=\ddbar \eta$ where
\begin{equation}\label{eta}
\eta(y,z) = -\frac{1}{4}\sum_{i,j=1}^{m-n} ({\rm Im}\,Z(y))^{-1}_{ij}(z_i - \bar{z}_i)(z_j - \bar{z}_j),
\end{equation}
and $Z:B\to\mathfrak{H}_{m-n}$ is a holomorphic period map from $B$ to the Siegel upper half space which was constructed in \cite{GTZ,He,HT}.

The form $\omega_{\rm SF}$ is called {\em semi-flat}. It was first introduced in \cite{GSVY} in the context of elliptically fibered $K3$ surfaces.

This result follows easily from the arguments of \cite{GTZ,He,HT}; in particular, $\omega_{\rm SF}$ is implicitly constructed in \cite{He} but without the explicit formula over small balls, while in \cite{GTZ,HT} we only considered the case when $N_0$ is a small ball. For the reader's convenience we give the proof.
\begin{proof}
We initially follow the construction in \cite[Section 3.2]{He}. For that construction to apply, we need the existence of a section (which we assume), and of a ``constant polarization'' (which exists because $\omega_M$ is K\"ahler, as in \cite[Proposition 2.1]{HT}). If $g_{y}$ denotes the unique flat K\"ahler metric on $M_y$ in the class $[\omega_M|_{M_y}]$, for any $y\in N_0$, then the restriction of $g_y$ to $T_{\sigma(y)}^{(1,0)}M_y$ defines a Hermitian metric on the holomorphic vector bundle $\mathcal{E}=\sigma^* T_{M_0/N_0}^{(1,0)}$ over $N_0$. As indicated above, we have a biholomorphism $M_0\cong \mathcal{E}/\Lambda$ for a holomorphic lattice bundle $\Lambda\subset \mathcal{E}$. Then $\Lambda$ induces a flat Gauss-Manin connection on $\mathcal{E}$, with horizontal space $\mathcal{H}$. Let $P:T^{(1,0)}M_0\to \mathcal{E}$ be the projection along $\mathcal{H}$, and for any $x\in M_0$, $u,v\in T_x^{(1,0)}M_0$ define
$g_{\rm SF}(u,v)=g_{f(x)}(Pu,Pv)$. Then $g_{\rm SF}$ defines a semipositive closed real $(1,1)$ form $\omega_{\rm SF}$ on $M_0$, as verified in \cite{He}, which restricts to
the correct flat K\"ahler metric on each fiber $M_y$.

If now $B\subset N_0$ is a coordinate ball, and $p:B\times\mathbb{C}^{m-n}\to f^{-1}(B)$ is the universal covering map, and assume that $p$ pulls back $\sigma$ to the zero section, then the construction in \cite{GTZ,HT} gives us
a function $\eta$ on $B\times\mathbb{C}^{m-n}$ defined by \eqref{eta}, such that $\ddbar\eta$ descends to a closed semipositive $(1,1)$ form $\omega'_{\rm SF}$ on $f^{-1}(B)$, with $\omega'_{\rm SF}|_{M_y}=\omega_{\rm SF}|_{M_y}$ for all $y\in B$. Both $\omega_{\rm SF}$ and $\omega'_{\rm SF}$ are invariant under translation by flat sections of the Gauss-Manin connection, and at every point on the zero section they are equal because they both vanish in the horizontal directions and are equal to the same flat K\"ahler metric on each fiber.
Therefore we conclude that $\omega_{\rm SF}=\omega'_{\rm SF}$ on all of $f^{-1}(B)$, as required. Uniqueness of $\omega_{\rm SF}$ follows from the fact that, locally on the base, it is given by this explicit formula, and that two different trivializations of $f^{-1}(B)$ which both map $\sigma$ to the zero section, must differ by fiberwise translation by a flat section, which leaves $\omega_{\rm SF}$ unchanged.
\end{proof}

As a consequence of the explicit formula \eqref{eta}, we see that if over a coordinate ball $B\subset N_0$ we define $\lambda_t:B\times\mathbb{C}^{m-n}\to B\times\mathbb{C}^{m-n}$ by $\lambda_t(y,z)=(y,t^{-\frac{1}{2}}z)$, then we have that
\begin{equation}\label{gnam}
t\lambda_t^*p^*\omega_{\rm SF}=p^*\omega_{\rm SF}.
\end{equation}

From now on we specialize to the setting of Theorem \ref{theorem1}, so that $M$ is projective hyperk\"ahler, and $f:M\to N$ is a holomorphic fiber space (and here we also assume the existence of a holomorphic section $\sigma:N_0\to M_0$, so that Theorem \ref{semiflat} applies), and $\ti{\omega}_t$ is the hyperk\"ahler metric on $M$ in the class $f^*[\omega_N]+t[\omega_N],$ $0<t\leq 1$. In this case it follows from \cite{Hw} that in fact $N\cong\mathbb{CP}^n$, but we will not need this fact. As before, we let $S'\subset N$ be the discriminant locus of $f$ and we set $N_0=N\backslash S, M_0=f^{-1}(N_0)$.

Using \eqref{gnam}, in \cite{HT} it is proven that given any compact $K\subset B\times\mathbb{C}^{m-n}$ and any $k\geq 0$, there is a constant $C$ independent of $t$ such that on $K$ we have
$$C^{-1}p^*(\omega_N+\omega_{\rm SF})\leq \lambda_t^*p^*\ti{\omega}_t\leq C(\omega_N+\omega_{\rm SF}),$$
and
\begin{equation}\label{cinfty}
\|\lambda_t^*p^*\ti{\omega}_t\|_{C^k(K,g_E)}\leq C,
\end{equation}
which is the same result as in \cite[Lemma 4.2 and Proposition 4.3]{GTZ} but without the need to use translations by holomorphic sections.

For later use, let us also assume that the fibration $f$ admits a holomorphic Lagrangian section $\sigma:N_0\to M_0$. As explained for example in \cite[p.43]{Fr}, \cite[Proposition 3.5]{Hw} or \cite[Proposition 2.4]{HO}), using this section and the holomorphic symplectic form on $M_0$, we get an isomorphism $M_0\cong T^{*(1,0)}N_0/\check{\Lambda},$ where $\check{\Lambda}\subset T^{*(1,0)}N_0$ is a lattice bundle, and let $p:T^{*(1,0)}N_0\to M_0$ be the natural projection (over any coordinate ball $B\subset N_0$ where the bundle is trivial, this agrees with the map $p$ as above). The dilations $\lambda_t$ are in fact well-defined as biholomorphisms $\lambda_t:T^{*(1,0)}N_0\to T^{*(1,0)}N_0$, which in trivializations over $B$ as above are given by the same formula as above.

We can now identify the smooth limit of the metrics $\lambda_t^*p^*\ti{\omega}_t$ on $T^{*(1,0)}N_0$ as $t\to 0$. The following is an improvement of \cite[Lemma 4.7]{GTZ}, again without the need to use translations by holomorphic sections:

\begin{proposition}\label{conver}
As $t\to 0$ we have
\begin{equation}\label{smooth}
\lambda_t^*p^*\ti{\omega}_t\to p^*(f^*\omega+\omega_{\rm SF}),
\end{equation}
smoothly on compact sets of $T^{*(1,0)}N_0$, where $\omega_{\rm SF}$ is given by Theorem \ref{semiflat}.
\end{proposition}
\begin{proof}
Fix any coordinate ball $B\subset N\backslash S'$.
Thanks to \cite[Proposition 3.1]{GTZ} we can find a holomorphic section $\ti{\sigma}:B\to f^{-1}(B)$ of $f$ and a smooth function $\xi$ on $f^{-1}(B)$ so that
$$\omega_M=T_{\ti{\sigma}}^*\omega_{\rm SF}+\ddbar \xi,$$
holds on $f^{-1}(B)$, where $T_{\ti{\sigma}}:f^{-1}(B)\to f^{-1}(B)$ is given by fiberwise translation by $\ti{\sigma}$.
We can now follow the argument of \cite[Lemma 4.7]{GTZ}, with some small modifications. Recall that we have
$$\ti{\omega}_t=f^*\omega_N+t\omega_M+\ddbar\vp_t,$$
with $|\vp_t|\leq C$ (see \cite{To1}).
On $p^{-1}(f^{-1}(B))\cong B\times\mathbb{C}^{m-n}$ we may then write
$$\lambda_t^*p^*\ti{\omega}_t=p^*f^*\omega+t\lambda_t^*p^*T_{\ti{\sigma}}^*\omega_{\rm SF}+\ddbar u_t,$$
where we have set $u_t=\vp_t\circ p\circ\lambda_t+t\xi\circ p\circ\lambda_t,$
and thanks to Theorem \ref{semiflat} we also have
$p^*\omega_{\rm SF}=\ddbar\eta,$
where $\eta$ is explicitly given by \eqref{eta}. The map $T_{\ti{\sigma}}$ is induced by the translation $(y,z)\mapsto (y,z+\ti{\sigma}(y))$ on $B\times\mathbb{C}^{m-n}$ (which we will also denote by $T_{\ti{\sigma}}$), and so this gives
$$t\lambda_t^*p^*T_{\ti{\sigma}}^*\omega_{\rm SF}=t\ddbar(\eta\circ T_{\ti{\sigma}}\circ \lambda_t),$$
and by direct inspection we see that this converges smoothly on compact sets to $p^*\omega_{\rm SF}$ as $t\to 0$.
Indeed, using \eqref{eta}, we see that
$$t(\eta\circ T_{\ti{\sigma}}\circ \lambda_t)(y,z)=t\eta(y,t^{-\frac{1}{2}}z+\ti{\sigma}(y))=\eta(y,z+t^{\frac{1}{2}}\ti{\sigma}(y)),$$ which converges to $\eta(y,z)$ smoothly on compact sets.
Thanks to \eqref{cinfty}, we see that $\ddbar u_t$ also has uniform local $C^\infty$ bounds, and since $|u_t|\leq C$, we conclude that the functions $u_t$ are themselves locally uniformly bounded in $C^\infty$. Arguing then exactly as in
\cite[Lemma 4.7]{GTZ} we conclude that $u_t\to \vp\circ f\circ p$ smoothly on compact sets, where $\omega_N+\ddbar\vp=\omega$ is the limiting metric on $N_0$. This concludes the proof.
\end{proof}
Of course, the same proof shows that even if $f:M_0\to N_0$ does not have a holomorphic Lagrangian section, we still obtain the convergence in \eqref{smooth} but just on the preimage of any coordinate ball $B\subset N_0$ (since on its universal cover $B\times\mathbb{C}^{m-n}$ we still have the stretching maps $\lambda_t$).

\subsection{The discriminant locus}

As in the Introduction, let $M$ be a projective hyperk\"ahler manifold, and $f:M\to N$ a surjective holomorphic map with connected fibers onto a compact K\"ahler manifold $N$ with $\dim N<\dim M$. Then we know by Matsushita \cite{Ma,Ma2} that $\dim N=\frac{1}{2}\dim M$ and that $f$ is an equidimensional holomorphic Lagrangian torus fibration, which in particular gives an algebraic completely integrable system where it is a submersion (see section \ref{4.2}). Also, by Hwang \cite{Hw} we have that $N$ is biholomorphic to $\mathbb{CP}^n$ where $\dim M=2n$.

If $S'\subset N$ denotes the discriminant locus of $f$, then it is well-known that $S'$ must be nonempty (see e.g. \cite[Proposition 4.1]{Hw}, where it is shown that $S'$ is in fact necessarily a divisor). We note here the following even stronger result, proved by the authors in \cite{TZ} with an extra hypothesis:

\begin{theorem}\label{triv}
Let $f:M\to N$ be a holomorphic submersion with connected fibers between compact K\"ahler manifolds with $c_1(M)=0$ in $H^2(M,\mathbb{R})$ (i.e. $M$ is a Calabi-Yau manifold). Then
$f$ is a holomorphic fiber bundle with fiber $F$ and base $N$ both Calabi-Yau manifolds.
\end{theorem}
\begin{proof}
It is well-known (see e.g. \cite{To4}) that $K_M$ is torsion in $\mathrm{Pic}(M)$, so there is a finite \'etale cover $\pi:\ti{M}\to M$ with $\ti{M}$ connected and with $K_{\ti{M}}$ trivial.
The composition $f\circ\pi:\ti{M}\to M\to N$ is a holomorphic submersion with possibly disconnected fibers, so we consider its Stein factorization $\ti{M}\overset{p}{\to}\ti{N}\overset{q}{\to}N$ where $\ti{N}$ is a (connected) compact K\"ahler manifold, $p$ is a holomorphic submersion with connected fibers and $q$ is a finite \'etale cover (see e.g. \cite[Lemma 2.4]{FG}). Therefore $p:\ti{M}\to\ti{N}$ satisfies the hypothesis of \cite[Theorem 1.3]{TZ}, and so it is a holomorphic fiber bundle with base $\ti{N}$ and fiber $\ti{F}$ both Calabi-Yau manifolds. By \cite[Lemma 4.5]{Fu}, it follows that $f:M\to N$ is a holomorphic fiber bundle as well. But we have finite unramified coverings $\ti{N}\overset{q}{\to} N$ and $\ti{F}\overset{\pi|_{\ti{F}}}{\to} F$, and so $N$ and $F$ are Calabi-Yau manifolds as well.
\end{proof}

\subsection{Estimates for the special K\"ahler metric  near the discriminant locus}

In this subsection we state our main estimate in the hyperk\"ahler setting, which implies the estimate \eqref{to0}. This estimate is then proved in section \ref{sectspec}.

As in the assumptions of Theorem \ref{theorem1}, let $M$ be a compact projective hyperk\"ahler $2n$-manifold, with holomorphic symplectic $2$-form $\Omega,$ and let $[\alpha]$ be an integral K\"{a}hler class on $M$. Assume that $M$ admits a surjective holomorphic map  $f: M \to N$  with connected fibers onto a compact K\"ahler $n$-manifold $N$. As before we let $S'\subset N$ be the discriminant locus of $f$, and let $N_0=N\backslash S',$ $M_0=f^{-1}(N_0)$.
The fibers $M_y=f^{-1}(y),y\in N_0,$ are holomorphic Lagrangian $n$-tori \cite{Ma,Ma2} so that $f:M_0\to N_0$ is an algebraic completely integrable system (see section \ref{4.2}). The base $N$ is known to be isomorphic to $\mathbb{CP}^n$ by \cite{Hw}, but we will not need this. The fibers $M_y,y\in N_0$, are Abelian varieties with the polarization $[\alpha_{y}]=[\alpha|_{M_{y}}]$, which is of type $(d_{1}, \cdots , d_{n})$ for some $d_i\in\mathbb{N}$. Let $\pi:\ti{N}\to N$ be a modification, which is an isomorphism over $N_0$, such that $E=\pi^{-1}(S')$ is a divisor with simple normal crossings, so near any point of $E$ there are coordinates $(u_{1}, \cdots, u_{n})$ on an open set $U\subset\ti{N}$ (which in these coordinates is the unit polydisc in $\mathbb{C}^n$) such that $E\cap U=\{u_{1}\cdots u_{k}=0\}$, for some $1\leq k\leq n$, and $\{u_i=0\}=E_{j_i}\cap U$  for some $1\leq j_i\leq\mu$ and all $1\leq i\leq k$.

By Section 3 of \cite{Fr}, there is a special K\"{a}hler metric $\omega$ on $N_{0}$ induced by the  algebraic  completely   integrable system $(f: M_{0} \rightarrow N_{0}, [\alpha],\Omega)$, and in \cite{GTZ2} we shows that this metric is equal to the collapsed smooth limit of the Ricci-flat metrics $\ti{\omega}_t$ as $t\to 0$, obtained in \cite{To1,GTZ}. The goal of this subsection is to study the asymptotic behaviour of $\pi^*\omega$ near $E$. As in section \ref{sectmain} we write $E=\bigcup_{j=1}^\mu E_j$ for the decomposition of $E$ into irreducible components.

\begin{theorem}\label{prop1} There are positive integers $m_i\in\mathbb{N}$, $1\leq i\leq\mu$, and a constant $C>0$ such that given any $y\in E$ and any local chart on $U$ as above, if we define the local uniformizing map $q:
\tilde{U} \rightarrow U$ ($\ti{U}$ is also the unit polydisc in $\mathbb{C}^n$) by
\begin{equation}\label{mapq}
q(t_1,\dots,t_n)=(t_1^{m_{j_1}},\dots,t_k^{m_{j_k}},t_{k+1},\dots,t_n),
\end{equation}
then on $\tilde{U}\backslash q^{-1}(E)$ we have
 $$q^{*} \pi^*\omega = \sqrt{-1} \sum_{i,j=1}^{n}g_{i\bar{j}}dt_{i}\wedge d\bar{t}_{j} $$  where
  \begin{equation}\label{estt}
  |g_{i\bar{j}}|\leq C(1-\varepsilon(i)\log |t_{i}|-\varepsilon(j)\log |t_{j}|), \  \ \ i,j=1, \cdots, n,
   \end{equation}
    where $\varepsilon(x)=1$ if $1\leq x \leq k$, and $\varepsilon(x)=0$ if $k+1\leq x \leq n$. Furthermore, the functions $g_{i\ov{j}}$ with $i,j>k$ extend continuously to all of $\ti{U}$. In particular, \eqref{to0} holds.
\end{theorem}

Since Theorem \ref{prop1} implies the estimate \eqref{to0}, once we complete the proof of Theorem \ref{prop1}, it will follow that parts (a) and (b) of Conjecture \ref{con1} hold, thanks to Theorem \ref{main}. We prove Theorem \ref{prop1} in the next section, which requires a detailed study of special K\"{a}hler metrics.

\subsection{The homeomorphism type of the Gromov-Hausdorff limit}
In this subsection we show how the estimates in Theorem \ref{prop1} also imply part (c) of Conjecture \ref{con1}. Therefore in the following we assume that we are in the setting of Theorem \ref{theorem1} and we assume the validity of Theorem \ref{prop1}, which will be proved in section \ref{sectspec}.

Let $(X,d_X)$ be the Gromov-Hausdorff limit of $(M,\ti{\omega}_t)$, which by Theorems \ref{main} and \ref{prop1} is isometric to the metric completion of $(N_0,\omega)$ and let $\psi: (N_{0}, \omega) \hookrightarrow (X, d_{X})$ be the isometric embedding. Our goal is to show that $X$ is homeomorphic to $N$ (which is of course homeomorphic to $\mathbb{CP}^n$).

Given any K\"ahler metric $\omega_N$ on $N$, the Yau Schwarz Lemma estimate $\ti{\omega}_{t} \geq C^{-1} f^{*}\omega_{N}$ proved in \cite{To1} gives a uniform Lipschitz constant bound for the map $f:(M, \ti{\omega}_{t})\rightarrow (N, \omega_{N}),$ independent of $t$ and so, up to passing to a sequence $t_i\to 0$, we obtain in the limit a Lipschitz surjective  map $h: (X, d_{X})\rightarrow (N, \omega_{N})$ with $h\circ \psi=\mathrm{Id}$.

Let also $\pi: \tilde{N}\rightarrow N$ be the modification (sequence of blowups with smooth centers) such that $E=\pi^{-1}(S')=\tilde{N}\backslash \pi^{-1}(N_{0})$ is a simple normal crossings divisor.
\begin{proposition}\label{bs}
There is a continuous surjective map $p: \tilde{N}\rightarrow X$ such that $\pi=h \circ p$.
\end{proposition}
\begin{proof}
 We have the homeomorphism $p=\psi\circ\pi: \pi^{-1}(N_{0}) \rightarrow \psi(N_{0})\subset X$. We claim that $p$ extends to the desired map $p: \tilde{N}\rightarrow X$. We use some of the notation as in the proof of Theorem \ref{main}.

To see this, let $y\in \tilde{N}\backslash \pi^{-1}(N_{0})$, and let $y_{i} \in \pi^{-1}(N_{0})$ be such that $y_{i} \rightarrow y$ in $\tilde{N}$. For any small neighborhood $\Delta^{n}(\rho)$ of $y$, we have $y_{i}\in\Delta^{n}(\rho)$ for $i$ sufficiently large, and for all $m\geq 0$ we have $$d_{\pi^*\omega} (y_{i},y_{i+m})\leq C \rho (\log \rho)^{d}\rightarrow 0.$$ Thus $\pi(y_{i})$ is a Cauchy sequence in $(N_{0}, \omega)$, and $\pi(y_{i})$ converges to a unique point $\bar{y}$ in $X$. We define $p(y)=\bar{y}$. The map $p$ is well-defined because if $y'_i$ is another sequence that converges to $y$ then for every small $\rho>0$ we see that for all $i$ sufficiently large $y_i,y'_i\in\Delta^{n}(\rho)$ and
$$d_{\pi^*\omega} (y_{i},y'_i)\leq C \rho (\log \rho)^{d}\rightarrow 0,$$
so that necessarily $\pi(y'_i)\to \bar{y}$ as well. A similar argument shows that $p$ is continuous. To see that $p$ is surjective, given any point $\bar{y}\in X$, there is a corresponding Cauchy sequence $y_i$ in $(N_0,\omega)$ which converges to it. Choose any preimages $y'_i\in\ti{N}, \pi(y'_i)=y_i.$ By compactness of $\ti{N}$, a subsequence will converge to some point $y'\in\ti{N}$, and we then have $p(y')=\bar{y}$.

If $y'=\pi (y) \in N$, then $\pi (y_{i}) =h(p(y_{i}))\rightarrow y'$, and thus $h(\bar{y})=y'$. This completes the proof.
\end{proof}

In particular, Proposition \ref{bs} already implies that $X$ is homeomorphic to $N$ if the discriminant locus $S'=N\backslash N_0$ of $f$ is a simple normal crossings divisor. Indeed, in this special case we may take $\ti{N}=N, \pi=\mathrm{Id}$, to conclude that $h\circ p=\mathrm{Id}$. Therefore $p$ must be injective, and it is then a continuous bijection between compact Hausdorff spaces, hence a homeomorphism.

To prove that $X$ is homeomorphic to $N$ in general, the key is the following:
\begin{proposition}\label{key}
Given any two points $y,\hat{y}\in\ti{N}$ such that $\pi(y)=\pi(\hat{y})$, we have that $p(y)=p(\hat{y})$.
\end{proposition}
Indeed, granted this, $p$ would then factor through $\pi$, i.e. we would find a continuous surjective map $\bar{p}:N\to X$ such that $p=\bar{p}\circ\pi$ ($\bar{p}$ is continuous because $\pi$ is a topological quotient map), and so $\pi=h\circ \bar{p}\circ\pi$. This clearly implies that $h\circ\bar{p}=\mathrm{Id}$ since $\pi$ is surjective, and so we conclude that $\bar{p}$ is a homeomorphism. It remains therefore to prove Proposition \ref{key}.
\begin{proof}
Recall that the map $\pi$ is a composition of blowups with smooth centers, so in particular its fibers are connected. Note that all the fibers of $\pi$ are subvarieties of $\ti{N}$, so their singular locus is itself stratified by locally closed smooth subvarieties of decreasing dimensions. We can then find a piecewise smooth path $\gamma$ in $\ti{N}$ joining $y$ and $\hat{y}$ such that $\pi(\gamma)=\pi(y)$ (the path may fail to be smooth only when crossing between the strata of the singularities of the fiber). In particular, $\gamma$ is contained in the simple normal crossings divisor $E=\ti{N}\backslash \pi^{-1}(N_0).$
Stratify $E$ by $E=E^1\supset E^2\supset\dots \supset E^n\supset E^{n+1}=\emptyset$ where each $E^j$ is the union of all $j$-tuple intersections of components of $E$, so that $E^j$ is smooth of codimension $j$ in $\ti{N}$ (usually disconnected). We can then find points $y_0=y, y_1,\dots, y_m=\hat{y}$ ($m\geq 1$) which all lie successively on the curve $\gamma$ such that for all $i=0,\dots,m-1$ the curve $\gamma$ between the points $y_i$ and $y_{i+1}$ (minus possibly the endpoints) is smooth and with image lying all in a unique open stratum $E^{k(i)}\backslash E^{k(i)+1}$. It is then enough to show that for all $i=0,\dots,m-1$ we have $p(y_i)=p(y_{i+1})$, which clearly implies what we want.

By renaming, we can just assume that we have two points $y,\hat{y}\in E$, with the curve $\gamma\subset E$ joining them which satisfies $\pi(\gamma)=\pi(y)=\pi(\hat{y})$, and except possibly for the endpoints, $\gamma$ is smooth with image contained in a unique open stratum $E^k\backslash E^{k+1}$ of codimension $k$ in $\ti{N}$, for some $1\leq k\leq n$.  Since $E^n$ is a finite set of points, the case $k=n$ is trivial (since by further subdivisions of the curve $\gamma$ we can always assume that the points $y,\hat{y}$ are close together), so we may assume that $k<n$, so $E^{k+1}$ is nonempty.
We then have three cases:\\

\noindent
{\bf Case 1. }Both $y,\hat{y}$ lie on $E^k\backslash E^{k+1}$.\\

We can assume that the points $y,\hat{y}$ are close together, because otherwise we can split the path $\gamma$ into smaller parts. This way, we can assume that the points $y,\hat{y}$ are close, both contained in the same chart $U$ as in Theorem \ref{prop1}. We can also assume that this chart is small enough so that after the basechange $q$ we can write $q^*\pi^*\omega=\ddbar\vp$ on $\ti{U}$ for some continuous psh function $\vp$ which is pulled back from $N$. The function $\vp$ is continuous thanks to Ko\l odziej's theorem \cite{Kol}, see the discussion in \cite[Section 4]{To1} for details.

So we are on a unit polydisc $U$ in $\mathbb{C}^n$ where $E$ is described by $t_1\cdots t_k=0$, and $E^k\cap U=\{t_1=\dots=t_k=0\}.$ After the basechange $q$, on $\ti{U}$ we have that
\begin{equation}\label{expr}
q^*\pi^*\omega=\sqrt{-1}\sum_{i,j=1}^n g_{i\ov{j}}dt_i\wedge d\ov{t}_j,
\end{equation}
where on $\ti{U}\backslash q^{-1}(E)$ the functions $g_{i\ov{j}}$ satisfy the properties in Theorem \ref{prop1}, in particular the functions
$g_{i\ov{j}}$
where all $i,j$ are $>k$, extend continuously to all the polydisc $\ti{U}$. Note that over $E^k$ the map $q$ is the identity, so we will abuse notation and denote the $q$-preimages of $y,\hat{y}$ and $\gamma$ by the same symbols.

Suppose that $\gamma$ is parametrized by the interval $[0,1]$, with $\gamma(0)=y,\gamma(1)=\hat{y}$. Define sequences of points $y'_i, \hat{y}'_i\in \ti{U}\backslash\{t_1\cdots t_k=0\}$ by translating $y,\hat{y}$ by $(1/i,\dots,1/i,0,\dots,0)$ (i.e. in the positive $t_1,\dots,t_k$ directions), so that $y'_i\to y, \hat{y}'_i\to \hat{y}$. Similarly, translate the path $\gamma$ in the same way to obtain a sequence of paths $\gamma_i$ contained in $\ti{U}\backslash\{t_1\cdots t_k=0\}$ which connect $y'_i$ and $\hat{y}'_i$, and with $\dot{\gamma}_i(s)=\dot{\gamma}(s)$ for all $0<s<1$ and all $i$. Since the path $\gamma$ is contained in $E^k$, its tangent vector lies in the span of $\frac{\de}{\de t_{k+1}},\dots,\frac{\de}{\de t_n}$ (and their conjugates), and so by Theorem \ref{prop1} the norm squared $|\dot{\gamma}_i(s)|^2_{q^*\pi^*\omega}$ (which only involves the coefficients $g_{i\ov{j}}$ with $i,j>k$ which extend continuously to $\ti{U}$) converges locally uniformly in $s$ to a limit, as $i\to\infty$.

We now claim that this limit is in fact zero. To prove this, recall that $q^*\pi^*\omega=\ddbar\vp$ where $\vp$ is pulled back from $N$, therefore it is constant along the path $\gamma$, since it is constant on the whole fiber $F$ of $\pi\circ q$ that contains $\gamma$. But the fiber $F$ is a complex subvariety of $\ti{U}$, and as we said earlier we may assume that the image $\gamma|_{(0,1)}$ is all inside an open (positive-dimensional) complex submanifold $F'\subset F\subset E^k$ (one of the elements of the stratification of $F$ by locally closed smooth subvarieties) along which $\vp$ is constant. Given any $0<s<1$, we may choose our coordinates as before so that near $\gamma(s)$ we can write $F'=\{t_1=\dots=t_\ell=0\}$ for some $k\leq\ell< n$. For each $i$ let $\iota_i:\Delta^{n-\ell}\to\ti{U}$ be the embedding of the unit polydisc in $\mathbb{C}^{n-\ell}$ given by $(t_{\ell+1},\dots,t_n)\mapsto (1/i,\dots,1/i,t_{\ell+1},\dots,t_n)$, whose image contains the image of $\gamma_i|_{(0,1)}$. Then the pullbacks $\iota_i^*q^*\pi^*\omega$ converge locally uniformly on $\Delta^{n-\ell}$ to a continuous limit $\omega_0$ as $i\to\infty$. At the same time, $\iota_i^*q^*\pi^*\omega=\ddbar (\vp\circ\iota_i),$ and the functions $\vp\circ\iota_i$ converge uniformly to the restriction of $\vp$ to $F'$, which is a constant. Therefore, $\iota_i^*q^*\pi^*\omega\to 0$ weakly as currents on $\Delta^{n-\ell}$, and hence $\omega_0=0$. It follows from this that $|\dot{\gamma}_i(s)|^2_{q^*\pi^*\omega}\to 0$, locally uniformly in $s$, proving our claim.

We can then take the image under $q$ to obtain sequences $y_i=q(y'_i),\hat{y}_i=q(\hat{y}'_i)$ of points in $U\backslash E$ which converge to $y,\hat{y}$, and are joined by the paths $q(\gamma_i)$ in $U\backslash E$ whose $\pi^*\omega$-length goes to zero. This proves that
$d_{\pi^*\omega}(y_i,\hat{y}_i)\to 0,$
and so $\pi(y_i)$ and $\pi(\hat{y}_i)$ converge to the same point in $X$, i.e. we have that $p(y)=p(\hat{y})$.\\

\noindent
{\bf Case 2. }The point $y$ lies on $E^k\backslash E^{k+1}$ and the point $\hat{y}$ lies on $E^{k+1}$.\\

Thanks to Case 1, we may assume that $y$ and $\hat{y}$ are very close together, so that they both lie in a chart $U$ where now $E\cap U=\{t_1\cdots t_{k+1}=0\}$, with $E^k\cap U=\{t_1=\dots=t_k=0\}, E^{k+1}\cap U=\{t_1=\dots=t_{k+1}=0\}$. Write $t,\hat{t}$ for the $(t_{k+2},\dots,t_n)$ coordinates of $y,\hat{y}$ respectively, so that in our coordinates we have
$y=(0,\dots,0,1,t), \hat{y}=(0,\dots,0,0,\hat{t})$ (up to scaling the $t_k$ coordinate).
For every $i\geq 2$, look at the point $\gamma(1-1/i)\in E^k\backslash E^{k+1}$.

Thanks to Case 1, there is $0<\ve_2\ll 1$ such that there is a path in $U\backslash E$ joining $(\ve_2,\dots,\ve_2,1,t)$ and $\gamma(1-1/2)+(\ve_2,\dots,\ve_2,0,\dots,0)$ ($k$ copies of $\ve_2$) of $\pi^*\omega$-length at most $1/2$. Then, choose $0<\ve_3<\ve_2$ such that there is a path in $U\backslash E$ joining $(\ve_3,\dots,\ve_3,1,t)$ and $\gamma(1-1/3)+(\ve_3,\dots,\ve_3,0,\dots,0))$ of $\pi^*\omega$-length at most $1/3$. Continue this way to obtain a sequence $\ve_i\to 0$ and paths $\gamma_i$ in $U\backslash E$ joining $y_i:=(\ve_i,\dots,\ve_i,1,t)$ and $\hat{y}_i:=\gamma(1-1/i)+(\ve_i,\dots,\ve_i,0,\dots,0)$ of $\pi^*\omega$-length at most $1/i$. This proves that $d_{\pi^*\omega}(y_i,\hat{y}_i)\to 0,$ and therefore $p(y)=p(\hat{y})$.\\

\noindent
{\bf Case 3. }Both $y,\hat{y}$ lie on $E^{k+1}$.\\

Recall that by construction we have that the path $\gamma$ (except the endpoints) is all contained in $E^k\backslash E^{k+1}$. Then we can find points $y'=\gamma(\ve),\hat{y}'=\gamma(1-\ve)\in E^{k}\backslash E^{k+1}$ which lie on $\gamma$, close to the corresponding $y,\hat{y}$ ($0<\ve\ll 1$).
Applying Case 2 we get that $p(y')=p(y)$ and $p(\hat{y}')=p(\hat{y})$, while Case 1 gives $p(y')=p(\hat{y}')$, and this completes the proof.
\end{proof}

\begin{remark}
In fact the smoothness of the base $N$ was not used substantially in our arguments, and Theorem \ref{theorem1} immediately extends to the case when $N$ is a possible singular projective variety. Indeed, in general the variety $N$ is the image of a map $f:M\to\mathbb{CP}^k$ for some $k$, and the Schwarz Lemma argument at the beginning of this section applies with $\omega_N$ replaced with the Fubini-Study metric on $\mathbb{CP}^k$, and so as before we obtain a Lipschitz map $h:(X,d_X)\to (\mathbb{CP}^k,\omega_{\rm FS})$ with image still equal to $N$, and this gives $h:X\to N$ with $h\circ\psi =\mathrm{Id}$ on $N_0$ (where now the complement $S'$ of $N_0$ in $N$ consists of the singular points of $N$ together with the critical values of $f$ in the smooth part of $N$). As before we have a map $\pi:\ti{N}\to N$, composition of blowups with smooth centers, such that $\ti{N}$ is smooth and $\pi^{-1}(S')=E$ is a divisor with simple normal crossings, and all the other arguments go through verbatim. However, since no example of such fibrations with $M$ hyperk\"ahler and $N$ singular is known (and it is conjectured that none exists, cf. \cite{Hw}), we have assumed for simplicity throughout the paper that $N$ is smooth.
\end{remark}
\section{Special K\"{a}hler geometry}\label{sectspec}
The goal of this section is to prove Theorem  \ref{prop1}, and therefore also Theorem \ref{theorem1} thanks to the results in sections \ref{sectmain} and \ref{secthk}. The proof  depends heavily on the  geometry of special K\"{a}hler metrics.
The notion of a special K\"{a}hler metric was introduced by physicists   (cf. \cite{Fr,Da,Hi,St}), and an intrinsic definition was given in  \cite{Fr}. Special K\"{a}hler metrics  exist on the base of algebraic completely integrable systems, and conversely such metrics, at least locally, induce algebraic completely integrable systems (provided that they are integral, in a suitable sense). We review some background of special   K\"{a}hler geometry, following
 \cite{Fr} closely, and then  we  prove Theorem  \ref{prop1}.

\subsection{Special K\"{a}hler metrics}\label{4.1} Let $(N_{0}, \omega)$ be a (possibly noncompact) K\"{a}hler manifold of dimension $n$. A special K\"{a}hler structure is a real torsion-free flat connection $\nabla$ on $TN_0$ such that $$\nabla\omega=0,\quad d^{\nabla}I=0,$$ where $I$ is the complex structure of $N_{0}$.

For a special K\"{a}hler manifold  $(N_{0}, \omega)$,  it is shown in  \cite{Fr}  that $N_{0}$ admits local flat Darboux coordinates, i.e. for any point $y\in N_{0}$, there are real  coordinates $y_{1}, \cdots, y_{2n}$ on a neighborhood of $y$ such that $\nabla dy_{i}=0$, $i=1, \cdots, 2n$, and  \begin{equation}\label{eq6.1}\omega=\sum\limits_{i=1}^{n}dy_{i}\wedge dy_{i+n}. \end{equation} The transition functions between two such coordinates are of the form $y_{j}'=\sum\limits_{i=1}^{2n}A_{ji}y_{i}+b_{j}$, $b_{j}\in \mathbb{R}$, and $A=(A_{ji})\in Sp(2n, \mathbb{R})$.  Hence the local flat Darboux coordinates covering gives a real affine manifold structure on $N_{0}$. If we have $A\in Sp(2n, \mathbb{Z})$ then we call it an integral special K\"ahler manifold.

If $y_{1}, \cdots, y_{2n}$ are local flat Darboux coordinates, there are two holomorphic coordinates systems $\{w_{1}, \cdots, w_{n}\}$ and $\{w_{1}^{*}, \cdots, w_{n}^{*}\}$  satisfying that \begin{equation}\label{eq6.2} dy_{i}= {\rm Re}dw_{i}, \  \  \  \  dy_{i+n}= -{\rm Re}dw_{i}^{*},  \  \  \  i=1, \cdots, n.\end{equation}   We  call $\{w_{i}\}$ the special coordinates system  and $\{w_{i}^{*}\}$ the conjugate coordinates system.  We define a complex matrix $Z=[Z_{ij}]$ by \begin{equation}\label{eq6.3} Z_{ij}= \frac{\partial w_{j}^{*}}{\partial w_{i}}.\end{equation} The K\"{a}hler form $\omega$ being a $(1,1)$-form implies that $Z_{ij}=Z_{ji}$, and there is a holomorphic function $\mathfrak{F}$, called a holomorphic prepotential function, such that  $ w_{i}^{*}=\frac{\partial \mathfrak{F}}{\partial w_{i}}$,  $  Z_{ij}= \frac{\partial^{2} \mathfrak{F}}{\partial w_{i}\partial w_{j} }.$
The K\"{a}hler potential is given by
$$\phi=\frac{1}{2}{\rm Im}\left(\sum\limits_{i=1}^{n}w_{i}^{*}\bar{w}_{i}\right),$$ and the K\"{a}hler metric is \begin{equation}\label{eq6.4}\omega=\sqrt{-1}\partial \overline{\partial} \phi=\frac{\sqrt{-1}}{2}\sum_{ij}{\rm Im} Z_{ij}dw_{i}\wedge d\bar{w}_{j}.  \end{equation} Thus $Z$ satisfies  the Riemann relations
\begin{equation}\label{eq6.5}Z^{T}=Z,  \  \  \  {\rm Im}Z >0, \end{equation} i.e. $Z$ belongs to the Siegel upper half space $\mathfrak{H}_n$.

If $g$ denotes the corresponding Riemannian metric of $(\omega, I)$, then $g$ is an affine K\"{a}hler metric with respect to the  local flat Darboux coordinates $y_{1}, \cdots, y_{2n}$  (cf. \cite{Fr,Hi}), i.e.   \begin{equation}\label{eq6.51}g= \sum_{ij}\frac{\partial^{2} \phi}{\partial y_{i} \partial y_{j}}dy_{i} dy_{j}.    \end{equation}   Furthermore, $g$ is a Monge-Amp\`ere metric, i.e. the potential function $\phi$ satisfies the real Monge-Amp\`ere equation $$ \det \left(\frac{\partial^{2} \phi}{\partial y_{i} \partial y_{j}}\right) \equiv {\rm const}, $$ since  $$\sqrt{\det (g_{ij})}dy_{1}\wedge \cdots\wedge dy_{2n}=\frac{1}{n!}\omega^{n}.$$

In \cite{Lu}, it is proved that if $(N_{0}, \omega)$ is complete, then $\omega$ is a flat metric (see also \cite{Co}),  which can also be obtained by using Cheng-Yau's theorem for Monge-Amp\`ere metrics in the case when $N_{0}$ is  compact  \cite[Corollary 2.3]{CY}.   Therefore, many interesting examples of special K\"{a}hler manifolds are not complete, and it is a natural question to study their completions.

\subsection{Algebraic completely integrable systems}\label{4.2}
An algebraic completely  integrable system is  a holomorphic Lagrangian fibration  $f: M_{0} \rightarrow N_{0}$ from a quasiprojective manifold   $M_{0}$ with $\dim_{\mathbb{C}}M_0=2n$, to a K\"{a}hler manifold  $ N_{0}$ with $\dim_{\mathbb{C}}N_0=n$, i.e.  $M_{0}$ admits  a holomorphic symplectic form $\Omega$ and the class $[\alpha]$ of an ample line bundle, $f:M_0\to N_0$ is a proper holomorphic submersion with connected fibers, and every fiber $M_{y}=f^{-1}(y)$, $y\in N_{0}$, is a complex Lagrangian submanifold with respect to $\Omega$. This forces every fiber $M_{y}$ to be an Abelian variety (without a specified origin) with the polarization $[\alpha_{y}]=[\alpha|_{M_{y}}]$, and the polarization is of type $(d_{1}, \cdots , d_{n})$, for some $d_{i}\in\mathbb{N}$.

It is shown in \cite[Section 3]{Fr} or \cite[Theorem 2]{Hi} that an integral special K\"{a}hler structure exists  on the base of  an algebraic completely   integrable system, and more   generally on the moduli space of holomorphic Lagrangian submanifolds in hyperk\"{a}hler manifolds (see also \cite{mar}). We recall the construction.

The vector bundle $E=f_*(T_{M_0/N_0}^{(1,0)})$ is isomorphic to $T^{*(1,0)}N_0$ via the pairing induced by $\Omega$, and the fiberwise action of $E$ on $M_0$ by exponentiation has as kernel a lattice subbundle $\check{\Lambda}\subset T^{*(1,0)}N_0$, with fiber $\check{\Lambda}_{y}\cong H_{1}(M_{y}, \mathbb{Z})$  for any $y\in N_{0}$, i.e. $\check{\Lambda}=\hom_{\mathbb{Z}}(R^{1}f_{*}\mathbb{Z},\mathbb{Z})$. The quotient $T^{*(1,0)}N_{0}/\check{\Lambda}$ is called the Jacobian family of $f$ (see e.g. \cite[Section 2.1]{Cla}, \cite[Section 3]{Fr}, \cite[Proposition 2.1]{Mar}), and is also a holomorphic Lagrangian fibration from a hyperk\"ahler quasiprojective manifold (with a polarization which is induced by $[\alpha]$ on $M_0$) with fibers isomorphic to those of $f$ (as polarized Abelian varieties) for all $y\in N_0$. The Jacobian fibration comes with a holomorphic Lagrangian section, and over every coordinate ball $B\subset N_0$ the original family $f$ also admits a holomorphic Lagrangian section, and therefore is isomorphic to the Jacobian family over any such $B$ (cf. \cite[p.43]{Fr}, \cite[Proposition 3.5]{Hw}, \cite[Proposition 2.4]{HO}), but there may be no such isomorphism globally over $N_0$ precisely because $f$ need not have a holomorphic Lagrangian section on all of $N_0$.

The special K\"ahler metric induced by $f$ is then defined purely using the Jacobian family as follows. There is a  canonical holomorphic symplectic form $\Omega_{can}$ on the  holomorphic cotangent bundle $T^{*(1,0)}N_{0}$, which  is characterized by $d\zeta=-\zeta^{*}\Omega_{can}$   for any 1-form $\zeta$.
Under  a local trivialization $ T^{*(1,0)}N_{0}|_{U}\cong U\times \mathbb{C}^{n}$ by $\sum\limits_{i} z_{i}dw_{i} \mapsto (w_{1}, \cdots, w_{n}, z_{1}, \cdots, z_{n})$, over some open subset $U\subset N_0$, we have  \begin{equation}\label{eq6.006}\Omega_{can}=\sum\limits_{i=1}^{n} d w_{i}\wedge dz_{i}.\end{equation}
Any local section of  $\check{\Lambda}$ is holomorphic  Lagrangian with respect to   $\Omega_{can}$.

Since $T^{*(1,0)}N_{0}$ is canonically isometric to the tangent bundle $T^{*}N_{0}$ as real smooth vector bundles, we have that  $\check{\Lambda}\otimes_{\mathbb{Z}}\mathbb{R}\cong T^{*}N_{0}$, and the embedding $\check{\Lambda}\hookrightarrow T^{*(1,0)}N_{0}$ is given by $\check{v}\mapsto \check{v}^{(1,0)}= \check{v}-\sqrt{-1}I\check{v}$.  The dual lattice bundle of $\check{\Lambda}$ is $\Lambda=R^{1}f_{*}\mathbb{Z}\cong \hom_{\mathbb{Z}}(\check{\Lambda},\mathbb{Z})$, and is a lattice subbundle of the tangent bundle $TN_{0}$.  The torsion-free flat connection $\nabla$ in the special K\"{a}hler structure is the connection such that any local section of $\Lambda$ is parallel.

Let $\omega_{\rm SF}$ be the  $(1,1)$-form  on $T^{*(1,0)}N_{0}/\check{\Lambda}$ constructed in Theorem \ref{semiflat}, so that the restriction $\omega_{\mathrm{SF},y}=\omega_{\rm SF}|_{M_{y}}$ on $T^{*(1,0)}_{y}N_{0}/\check{\Lambda}_{y}\cong M_{y}$ is the flat K\"{a}hler metric representing $[\alpha_{y}]$. Therefore $\omega_{\rm SF}$ is  a  real bundle  symplectic form on $T^{*(1,0)}N_{0}$,  i.e. $\omega_{\mathrm{SF},y}$ is a linear symplectic form on $T^{*(1,0)}_{y}N_{0}$.  Since $\omega_{\rm SF}$ represents an integral cohomology class on the fibers, it defines an integral   symplectic form on $\check{\Lambda}$.
If $\check{v}_{1}, \cdots, \check{v}_{2n}$ are local sections  of $\check{\Lambda} $, which are  symplectic basis with respect to $\omega_{\rm SF}$, i.e.
\begin{equation}\label{eq4.2++}\omega_{\rm SF}(\check{v}_{i},\check{v}_{j})=0, \  \  \omega_{\rm SF}(\check{v}_{i+n},\check{v}_{j+n})=0 \  \  {\rm  and } \  \  \omega_{\rm SF}(\check{v}_{i},\check{v}_{j+n})=d_{i}\delta_{ij},\end{equation} $1\leq i,j\leq n$, then
local flat Darboux coordinates $y_{1}, \cdots, y_{2n}$ are defined such that $dy_{i}=d_{i}^{-1}\check{v}_{i}$ and $dy_{i+n}=-\check{v}_{i+n}$, $i=1, \cdots, n$. Here we regard  $\check{\Lambda}$ as a lattice subbundle of $T^{*}N_{0}$, and $\check{v}_{i}$, $i=1, \cdots,2n$,  as real 1-forms.

  We have  the special coordinates $w_{1}, \cdots, w_{n}$ and the conjugate  coordinates $w_{1}^{*}, \cdots, w_{n}^{*}$, and  $Z_{ij}= \frac{\partial w_{j}^{*}}{\partial w_{i}}$.    The image of $\check{\Lambda}_{y} \hookrightarrow T^{*(1,0)}_{y}N_{0} $ is generated by $dw_{i}$ and $dw_{i}^{*}$, $i=1, \cdots, n$, and the period matrix  of the Abelian variety $M_{y}$ is $[\Delta_{d},  Z_{ij}(y)]$, where $\Delta_{d}={\rm diag} (d_{1}, \cdots , d_{n})$.
By (\ref{eq6.1}) and (\ref{eq6.4}),
the special K\"{a}hler form is given  as   \begin{equation}\label{eq6.06}\omega =\sum\limits_{i=1}^{n}dy_{i}\wedge dy_{i+n}= \frac{\sqrt{-1}}{2}\sum_{ij}{\rm Im} Z_{ij}dw_{i}\wedge d\bar{w}_{j}.  \end{equation}

We end this subsection by showing  the following local calculation of the Ricci curvature of the special K\"ahler metric. A more general result is proved by the first-named author in \cite[Proposition 4.1]{To1} (see also \cite{ST}) in the case when $M\to N$ is a holomorphic fiber space with $M$ Calabi-Yau, and $N$ compact.

\begin{proposition}\label{pro0}
The Ricci form of the special K\"{a}hler metric $\omega$ is  the Weil-Petersson form $\omega_{WP}$  of the family of Abelian varieties $f: M_{0} \rightarrow N_{0}$, i.e. $${\rm Ric}(\omega)= \omega_{WP}.$$
\end{proposition}

\begin{proof} Locally on $N_0$ we have $${\rm Ric}(\omega)=-\sqrt{-1}\partial\overline{\partial}\log\det ({\rm Im} Z_{ij}).  $$ For a sufficiently small open subset $U\subset N_{0}$, $ f^{-1}(U)\cong U\times \mathbb{C}^{n}/\check{\Lambda}_{U} $, where $\check{\Lambda}_{U} ={\rm Span}_{\mathbb{Z}}({\rm diag} (d_{1}, \cdots , d_{n})| Z_{ij}) $,  and $f$ is the projection  to $U$. We denote $z_{1}, \cdots, z_{n}$ the coordinates on $\mathbb{C}^{n}$, and $\Theta=d z_{1}\wedge \cdots \wedge d z_{n}$.   The Weil-Petersson form (see e.g. \cite{To1,ST}) is defined by $$\omega_{WP}=-\sqrt{-1}\partial\overline{\partial}\log \int_{M_{y}}(-1)^{\frac{n^{2}}{2}} \Theta \wedge \overline{\Theta}=-\sqrt{-1}\partial\overline{\partial}\log V_{y},  $$ where $V_{y}$ is the Euclidean volume of $M_{y}\cong \mathbb{C}^{n}/\check{\Lambda}_{U,y}$.
 We obtain the conclusion by $V_{y}=\det ({\rm Im} Z_{ij}) \prod\limits_{k}d_{k} $.
\end{proof}

\subsection{Estimates from degenerations of Hodge structures}
Now we are ready to  prove Theorem \ref{prop1}.

\begin{proof}[Proof of Theorem \ref{prop1}]
 Let $\omega$  be the  special K\"{a}hler metric  on $N_{0}$ induced by the   algebraic  completely   integrable system $(f: M_{0} \rightarrow N_{0}, [\alpha], \Omega)$, as explained in previous subsection. Recall that $f$ comes from a map $f:M\to N$ as in the Introduction, with $M$ projective hyperk\"ahler, that $N_0=N\backslash S'$, where $S'$ is the discriminant locus of $f$, and that $\pi:\ti{N}\to N$ is a modification with $E=\pi^{-1}(S')$ a simple normal crossings divisor.

 Let $\nabla$ be the flat connection of the special K\"{a}hler structure, and let  $\check{\Lambda}$ and $\Lambda$ be the lattice bundles as in last subsection, i.e. $\Lambda \cong R^{1}f_{*}\mathbb{Z} $ and $\check{\Lambda} \cong \hom_{\mathbb{Z}}(\Lambda,\mathbb{Z})$.
   We have the canonical identifications $\check{\Lambda}_{\mathbb{R}}=\check{\Lambda}\otimes_{\mathbb{Z}}\mathbb{R} \cong T^{*}N_{0}$,   $\Lambda_{\mathbb{R}}=\Lambda\otimes_{\mathbb{Z}}\mathbb{R} \cong R^{1}f_{*}\mathbb{R}\cong TN_{0} $ and $H^{1,0}(M_{y})\cong T_{y}^{(1,0)}N_{0} $ for any $y\in N_{0}$.
 There is a  weight-one integral polarized variation of Hodge structures on $N_0$ given by the quadruple $$(\mathcal{F}^{1}=T^{(1,0)}N_{0}\subset \mathcal{F}^{0}=TN_{0}\otimes_{\mathbb{R}}\mathbb{C}, \Lambda\subset TN_{0}, \nabla, \omega)$$ (in fact this data is equivalent to an integral special K\"ahler structure, cf. \cite[8.4]{Da}, \cite[Section 3.3]{Her} or \cite[Section 3]{BM}). Since $\pi:\ti{N}\backslash E\to N_0$ is a biholomorphism, we can also view this as a variation of Hodge structures on $\ti{N}\backslash E$. To alleviate notation, we will denote $\pi^*\omega$ simply by $\omega$.

 Let $U$ be a local chart   near a point of $E$, and $\{t'_{1}, \cdots, t'_{n}\}$ be coordinates on $U$ such that  $E\cap U$ is given by $\{t'_{1}\cdots t'_{k}=0\}$, and $\{t'_i=0\}=E_{j_i}\cap U$ for some $1\leq j_i\leq\mu,$ and all $1\leq i\leq k$. For a fixed $y\in U\backslash E$,  the lattice bundles  $\check{\Lambda}$ and $\Lambda$ induce  monodromy representations $\check{\rho}: \pi_{1}(U\backslash E) \rightarrow {\rm Aut}\check{\Lambda}_{y}$ and  $\rho: \pi_{1}(U\backslash E) \rightarrow {\rm Aut}\Lambda_{y}$  respectively.
   For $\gamma \in \pi_{1}(U\backslash E)$, we denote  $T_{\gamma}: \Lambda_{y}\rightarrow \Lambda_{y}$ and $\check{T}_{\gamma}: \check{\Lambda}_{y}\rightarrow \check{\Lambda}_{y}$  the corresponding  monodromy operators of $\rho$ and $\check{\rho}$.    If $\langle\cdot,\cdot\rangle$ denotes the  pairing between $\check{\Lambda}$ and $\Lambda$, we have $$\langle \check{T}_{\gamma} \alpha,\beta\rangle=\langle\alpha,T_{\gamma^{-1}}\beta\rangle$$ for any $\alpha \in  \check{\Lambda}_{y}$ and $\beta \in  \Lambda_{y}$ by $\langle \check{T}_{\gamma} \alpha,\beta\rangle=\langle \check{P}_{\gamma^{-1}(t)}\check{T}_{\gamma} \alpha,P_{\gamma^{-1}(t)}\beta\rangle$, $t\in [0,1]$,  where  $P_{\gamma^{-1}(t)}$ and  $\check{P}_{\gamma^{-1}(t)}$ are parallel transports  of $\Lambda$ and $\check{\Lambda}$ respectively.

 \begin{lemma}\label{le01}  There are positive integers $m_i\in\mathbb{N},$ $1\leq i\leq\mu$, such that the following holds. Let $q:
\tilde{U} \rightarrow U$  be the branched covering given by  $$
q(t_1,\dots,t_n)=(t_1^{m_{j_1}},\dots,t_k^{m_{j_k}},t_{k+1},\dots,t_n),
$$ where $t_{1}, \cdots, t_{n}$ are coordinates on $\tilde{U}$. We
 denote still by $\Lambda$ and $\check{\Lambda}$ the pullbacks of the respective lattice bundles via $q$. Then we have
 \begin{itemize}
   \item[i)] For any $\gamma \in \pi_{1}(\tilde{U}\backslash q^{-1}(E))$,   the monodromy operators of $ \Lambda$ and $ \check{\Lambda}$, still denoted by $T_{\gamma}$ and $\check{T}_{\gamma}$, satisfy $$(T_{\gamma}-I)^{2}=0,  \  \  \  \  (\check{T}_{\gamma}-I)^{2}=0.$$
     \item[ii)]  There are multivalued sections $\check{v}_{1}, \cdots , \check{v}_{2n}$ of $ \check{\Lambda}\subset T^{*} (\tilde{U}\backslash q^{-1}(E))$ such that  $$q^{*}\omega=-\sum_{i=1}^{n} d_{i}^{-1}  \check{v}_{i}\wedge \check{v}_{n+i}.$$
          \item[iii)]  For any $\gamma \in \pi_{1}(\tilde{U}\backslash q^{-1}(E))$,  $\check{T}_{\gamma}\check{v}_{j}=\check{v}_{j}$,    $j=1, \cdots, n$.  Hence  $\check{v}_{1}, \cdots , \check{v}_{n}$ are in fact single-valued sections of $\check{\Lambda}$  over $\tilde{U}\backslash q^{-1}(E)$.
    \end{itemize}
 \end{lemma}

 \begin{proof} By shrinking $U$ if necessary, we may assume that  $\pi_{1}(U\backslash E)\cong \mathbb{Z}^{k}$, generated by loops which wrap once around each component of $E$ that intersects $U$.   The Monodromy theorem (cf. Chapter II of  \cite{Gr}) shows that for every $\gamma\in \pi_{1}(U\backslash E)$ the eigenvalues of $T_{\gamma}$ are roots of unity, and if we take $\gamma$ to be a loop which wraps once around the irreducible  component $E_i$ of $E$  then we get one eigenvalue which is an $m_i^{th}$ root of unity (and the others equal $1$), with $m_i$ being independent of which particular loop we choose (once we fix its orientation). This way we obtain the positive integers $m_i$, $1\leq i\leq \mu$, and we can locally define the branched covering $q:\tilde{U} \rightarrow U$ as in \eqref{mapq},  such that for any $\gamma \in \pi_{1}(\tilde{U}\backslash q^{-1}(E))$,   $(T_{\gamma}-I)^{2}=0$.   By  duality, $(\check{T}_{\gamma}-I)^{2}=0$, and thus both $T_{\gamma}$ and $\check{T}_{\gamma}$ are unipotent.

 If we  let $\mathcal{N}_{\gamma}=\log T_{\gamma}=T_{\gamma}-I$, then $\mathcal{N}_{\gamma}$  is an element in  the Lie algebra $$\mathfrak{g}_{0}=\{\mathcal{N}\in \hom_{\mathbb{R}} (\Lambda_{\mathbb{R},y},\Lambda_{\mathbb{R},y})| \omega (\mathcal{N} \cdot, \cdot)+ \omega (\cdot, \mathcal{N} \cdot)=0\}$$ (cf.  Chapter V of  \cite{Gr}). Let $\gamma_{i}$, $i=1, \cdots, k$, be the generators of  $\pi_{1}(\tilde{U}\backslash q^{-1}(E))$, and
denote   $T_{i}=T_{\gamma_{i}}$,
$\mathcal{N}_{i}=\log T_{i}$, $i=1, \cdots, k$.
The monomdromy cone is the open cone  $$\Sigma=\left\{\sum_{i} \nu_{i}\mathcal{N}_{i}|  \nu_{i}>0 \right\}\subset \mathfrak{g}_{0}. $$ Note that for any $\mathcal{N}\in \Sigma$, we have $\mathcal{N}^{2}=0$.   If $\mathcal{N}=\sum\limits_{i} \nu_{i}\mathcal{N}_{i}$, then $\mathcal{N}=\log \prod\limits_{i} T_{i}^{\nu_{i}}$. Since  $T=\prod\limits_{i} T_{i}^{\nu_{i}}$ is the monodromy operator of  the curve $\prod\limits_{i} \gamma_{i}^{\nu_{i}}$, we have $\mathcal{N}^{2}=(T-I)^{2}=0$.

For any $\mathcal{N}\in \Sigma$, we have  $W_{0}={\rm Im} \mathcal{N}=\oplus_{i=1}^{k} {\rm Im} \mathcal{N}_{i}$ and $W_{1}=\ker \mathcal{N}=\bigcap \limits_{i=1}^{k} \ker \mathcal{N}_{i}$ by \cite[Lemma 2.2]{CCK}, which give the
 monodromy weight filtration  $$ \{0\}\subset W_{0} \subset W_{1} \subset W_{2}=\Lambda_{\mathbb{R},y}. $$   Then $W_{0}$ is an isotropic subspace and $$W_{1}=W_{0}^{\bot}=\{u\in \Lambda_{\mathbb{R},y}| \omega(u,v)=0, \forall v \in W_{0}\},$$ by $\omega(\mathcal{N} \cdot,\cdot)+\omega(\cdot,\mathcal{N}\cdot)=0 $.   Thus we have a Lagrangian subspace $W_{0}\subset L \subset W_{1}$, and
  a symplectic basis $\{ v_{1}, \cdots, v_{2n}\}$ of $\omega$ such that  $\omega(v_{i},v_{j})=0$, $\omega(v_{i+n},v_{j+n})=0$ and  $\omega(v_{i},v_{j+n})=-d_{i}^{-1}\delta_{ij}$, $1\leq i,j\leq n$, and $L=\mathbb{R}\{v_{n+1}, \cdots, v_{2n}\}$.    By varying $y$ in $\tilde{U}\backslash q^{-1}(E)$, we regard   $v_{n+1}, \cdots, v_{2n}$ as (single-valued) sections of $\Lambda$ on $\tilde{U}\backslash q^{-1}(E)$ (since they lie in $W_1$ and so are monodromy invariant), and  $v_{1}, \cdots, v_{n}$ as multi-valued sections.
  If $\check{v}_{1}, \cdots, \check{v}_{2n}$ are dual multisections of $\check{\Lambda} $, we have
   $q^{*}\omega=-\sum\limits_{i} d_{i}^{-1}  \check{v}_{i}\wedge \check{v}_{n+i}$.

   Note that $T_{\gamma}v_{i+n}=v_{i+n}$, and $T_{\gamma}v_{i}- v_{i} \in W_{0}\subset L$, $i=1, \cdots, n$, for any $\gamma\in \pi_{1}(\tilde{U}\backslash q^{-1}(E))$. We have $\langle \check{T}_{\gamma}\check{v}_{i}, v_{j+n} \rangle= \langle \check{v}_{i}, v_{j+n} \rangle =0$, and $\langle \check{T}_{\gamma}\check{v}_{i}, v_{j} \rangle= \langle \check{v}_{i}, v_{j}+ (T_{\gamma^{-1}}v_{j}- v_{j}) \rangle =\delta_{ij}$, $1\leq i,j \leq n$.  Thus $\check{T}_{\gamma}\check{v}_{i}=\check{v}_{i}$ for any $\gamma$, and
    so $\check{v}_{1}, \cdots , \check{v}_{n}$ are single-valued sections of $\check{\Lambda}\subset T^{*} (\tilde{U}\backslash q^{-1}(E))$.
 \end{proof}

Note that $\Lambda_{\mathbb{R}}\cong T\tilde{U}\backslash q^{-1}(E)$, $\check{\Lambda}_{\mathbb{R}}\cong T^{*}\tilde{U}\backslash q^{-1}(E)$ and $\check{\Lambda}$ is identified as a lattice subbundle in $T^{*(1,0)}\tilde{U}\backslash q^{-1}(E)$ via $\check{v}\mapsto \check{v}^{(1,0)}= (1-\sqrt{-1}I)\check{v}$,  where $I$ is the complex structure of $\tilde{U}$.  The multisections $\check{v}_{j}^{(1,0)}$, $j=1, \cdots, 2n$, are therefore holomorphic (and those with $j=1,\cdots n$ are single-valued).

For $y\in U\backslash E$ the Abelian varieties $M_{y}\cong T^{*(1,0)}_{y}\tilde{U}/ \check{\Lambda}_{y}$, with $\check{\Lambda}_{y}\cong {\rm Span}_{\mathbb{Z}}(\check{v}_{1}^{(1,0)},\cdots, \check{v}_{2n}^{(1,0)} )$, have period matrices $[\Delta_{d}, Z]$, where $\Delta_{d}={\rm diag} (d_{1}, \cdots , d_{n})$, and as in the previous section $Z$ is a multi-valued holomorphic map from $U\backslash E$ to the Siegel upper half space $\mathfrak{H}_n$. Pulling back by $q$ we will also regard $Z$ as a multi-valued map from $\ti{U}\backslash q^{-1}(E)$, which has the property that
$\check{v}_{j+n}^{(1,0)}=\sum\limits_{i=1}^{n} Z_{ji}d_{i}^{-1}\check{v}_{i}^{(1,0)},$   $j=1, \cdots, n$. From Lemma \ref{le01} (ii) we then get
\begin{equation}\label{houzi}
q^{*}\omega=-\sum_{i=1}^{n} d_{i}^{-1}  \check{v}_{i}\wedge \check{v}_{n+i} =2\sqrt{-1}\sum_{i,j=1}^n d_{i}^{-1}d_j^{-1}{\rm Im}Z_{ij} \check{v}_{i}^{(1,0)}\wedge\ov{\check{v}_{j}^{(1,0)}}. \end{equation}

   The following lemma is a standard consequence of the  Nilpotent Orbit Theorem \cite[Theorem 4.12]{Sc}, see e.g. Chapter V of \cite{Gr}, and we present the brief  proof  for the sake of  completeness.

 \begin{lemma}\label{le02} There are rational matrices $\eta_{i}$, $i=1, \cdots, k$, and a holomorphic matrix valued function $Q(t)$ on $\tilde{U}$ such that   $$Z(t)=Q(t)+\sum_{i=1}^{k}\frac{\log t_{i}}{2\pi \sqrt{-1}}\Delta_{d}\eta_{i}.$$
Furthermore, there is a constant $\lambda>0$ such that all eigenvalues of $\mathrm{Im} Z(t)$ are bounded below by $\lambda$ as $t\in \tilde{U}\backslash q^{-1}(E)$ approaches $q^{-1}(E)$.
  \end{lemma}

\begin{proof}
   Note that   there are unique local sections $\theta_{1}, \cdots, \theta_{n}$ of $\mathcal{F}^{1}=T^{(1,0)}\tilde{U}\backslash q^{-1}(E)$ such that $\langle \theta_{i}, \check{v}_{j}\rangle =d_{i} \delta_{ij}$,  $ 1\leq i,j\leq n$,  and  equivalently $ \langle \theta_{i}, \check{v}_{j}^{(1,0)}\rangle = 2d_{i} \delta_{ij}$.   Hence $ \langle \theta_{i}, \check{v}_{j+n}^{(1,0)}\rangle = 2Z_{ji}$,  $\langle \theta_{i}, \check{v}_{j+n}\rangle = Z_{ji} $,  and  $$ \theta_{i}=d_{i} v_{i}+ \sum_{j=1}^{n}Z_{ij}v_{j+n}, \  \  \ i=1, \cdots, n.$$
Let $\mathfrak{D}$ be the classifying space of the polarized variation of Hodge structures $(T^{(1,0)}N_{0}\subset TN_{0}\otimes_{\mathbb{R}}\mathbb{C}, \Lambda, \omega)$, and
 $\mathcal{P}: \tilde{U}\backslash q^{-1}(E) \rightarrow \mathfrak{D}/ \Gamma$ be the the  period map, where $\Gamma$ is a discrete subgroup of $Sp(2n, \mathbb{R})$ depending on the polarization $\omega$. Note that
  $\mathfrak{D}$ can be identified as  the Siegel upper half space $$\mathfrak{H}_{n}=\{n\times n  \ {\rm complex  \ matrix} \ A |A^{T}=A,   \  \  {\rm Im}A >0\},$$ via $T^{(1,0)}N_{0}={\rm Span}_{\mathbb{C}}(\theta_{1}, \cdots, \theta_{n}) \mapsto Z(t)$.

The universal covering  $\bar{U}\rightarrow \tilde{U}\backslash q^{-1}(E) $ is given by $t_{i}=\exp 2\pi\sqrt{-1}q_{i}$, $i=1, \cdots, k$, and $t_{j}=q_{j}$, $j=k+1, \cdots, n$. If  $\tilde{\mathcal{P}}: \bar{U} \rightarrow \mathfrak{D}$ is the lifting of $\mathcal{P}$, we have $$ \tilde{\mathcal{P}}(\cdots,q_{i}+1, \cdots)=T_{i}\tilde{\mathcal{P}}(\cdots,q_{i}, \cdots),  \  \  \  \ i=1, \cdots, k.$$ Equivalently,   $$ \begin{bmatrix}
      \Delta_{d} ,Z(q_{i}+1)
     \end{bmatrix}=\begin{bmatrix}
      \Delta_{d} ,Z(q_{i})
     \end{bmatrix}\begin{bmatrix}
     I , \eta_{i} \\ 0, I
     \end{bmatrix},$$  where  $ \begin{bmatrix}
      I ,0           \\
      \eta^{T}_{i}, I
     \end{bmatrix}$  is the matrix of $T_{i}$ with  respect  to the basis $v_{1}, \cdots, v_{2n}$, which are  rational matrices,  and thus  $$   Z(q_{i}+1)=  Z(q_{i})+ \Delta_{d} \eta_{i} .$$
If we define  $$\tilde{\mathcal{Q}}=\exp\left(-\sum_{i=1}^{k}q_{i}\mathcal{N}_{i}\right)\tilde{\mathcal{P}},$$ where  $\mathcal{N}_{i}=\log T_{i}=T_{i}-I$,  then $\tilde{\mathcal{Q}}(\cdots,q_{i}+1, \cdots)=\tilde{\mathcal{Q}}(\cdots,q_{i}, \cdots)$, and $\tilde{\mathcal{Q}}$ descends to a holomorphic  map $\mathcal{Q}: \tilde{U}\backslash  q^{-1}(E)  \rightarrow \mathfrak{D}$.

   By the Schmid's  Nilpotent Orbit Theorem \cite[Theorem 4.12]{Sc} (see also Chapter V of  \cite{Gr}),  $\mathcal{Q}$ extends to a holomorphic  map $\mathcal{Q}: \tilde{U} \rightarrow \check{\mathfrak{D}}$, where as in \cite{Sc} $\check{\mathfrak{D}} $ denotes the compact dual of $\mathfrak{D}$, which implies that there is a symmetric-matrices valued  holomorphic function $Q(t)$ on $\tilde{U} $ such that $$  \begin{bmatrix}
      \Delta_{d} , Q(t)
     \end{bmatrix}=\begin{bmatrix}
      \Delta_{d} , Z(t)
     \end{bmatrix}\prod_{i=1}^{k}\begin{bmatrix}
     I , - q_{i}\eta_{i} \\ 0, \  \   \  I
     \end{bmatrix}. $$
      Thus $$Z(t)=Q(t)+ \sum_{i=1}^{k}\frac{\log t_{i}}{2\pi \sqrt{-1}}\Delta_{d}\eta_{i}.$$
To estimate the eigenvalues of $\mathrm{Im}Z(t)$, we note as in \cite[Chapter V]{Gr} that each of the matrices $\eta_i$ is of the form
$\begin{bmatrix} \eta_i^1 \ 0\\ 0 \ \ 0\end{bmatrix}$ with $\eta_i$ a $\nu\times\nu$ positive definite symmetric matrix, where $\nu=\dim W_0\leq n$. Positive definiteness follows from the Nilpotent Orbit Theorem, which also implies that the imaginary part of
$$Q(0,\dots,0,t_{k+1},\dots,t_n)+\sum_{i=1}^{k}\frac{\log t_{i}}{2\pi \sqrt{-1}}\Delta_{d}\eta_{i},$$
is strictly positive definite (say with smallest eigenvalue $2\lambda>0$) whenever $t_1,\dots,t_k$ are sufficiently small while $t_{k+1},\dots,t_n$ remain in a fixed small polydisc. From this, and the fact that $Q$ is holomorphic on all of $\ti{U}$, it follows easily that the smallest eigenvalue of $\mathrm{Im} Z(t)$ is at least $\lambda$ as $t$ approaches $q^{-1}(E)$.
      \end{proof}

\begin{lemma}\label{extend}
The sections $\check{v}^{(1,0)}_{1}, \cdots , \check{v}^{(1,0)}_{n}$ of $\check{\Lambda}$ over $\tilde{U}\backslash q^{-1}(E)$ extend to holomorphic $1$-forms on $\ti{U}$.
\end{lemma}

\begin{proof}
Writing $\check{v}^{(1,0)}_{j}(t)=\sum_{k=1}^n h_j^k(t)dt_k$ for some holomorphic functions $h_j^k(t)$ on $\ti{U}\backslash q^{-1}(E)$, \eqref{houzi} gives us
\[\begin{split}
q^{*}\omega&=2\sqrt{-1}\sum_{i,j,k,\ell=1}^n d_{i}^{-1}d_j^{-1}{\rm Im}Z_{ij}(t) h_i^k(t) \ov{h_j^\ell(t)}dt_k\wedge d\ov{t}_\ell\\
&\geq C^{-1}\sqrt{-1}\sum_{i,k,\ell=1}^nh_i^k(t) \ov{h_i^\ell(t)}dt_k\wedge d\ov{t}_\ell,
\end{split}\]
using Lemma \ref{le02}. Since the coefficients of the closed positive current $q^*\omega$ on $\ti{U}$ are Radon (complex) measures, they have locally finite mass near any point in $q^{-1}(E)$, and so by the previous inequality all the holomorphic functions $h_i^k$ are $L^2$ integrable on $\ti{U}\backslash q^{-1}(E)$ (up to shrinking $\ti{U}$), and hence they extend holomorphically across $q^{-1}(E)$ (see e.g. \cite[Proposition 1.14]{Oh}).
\end{proof}
  Since $\check{v}_{j}^{(1,0)}$, $j=1, \cdots, n$, are holomorphic  Lagrangian sections of $T^{*(1,0)}\tilde{U}\backslash q^{-1}(E)$ with respect to the canonical holomorphic  symplectic form $\Omega_{can}$,   the  extensions $\check{v}_{1}^{(1,0)}, \cdots, \check{v}_{n}^{(1,0)}$  are holomorphic   Lagrangian sections of $T^{*(1,0)}\tilde{U}$.  Then $d \check{v}_{j}^{(1,0)}= (\check{v}_{j}^{(1,0)})^{*}\Omega_{can}=0$, $j=1, \cdots, n$,  (where the $(1,0)$ form $\check{v}_j^{(1,0)}$ gives a holomorphic map $\ti{U}\to T^{*(1,0)}\tilde{U}$, under which we pull back the canonical holomorphic symplectic form), and   there are holomorphic functions $w_{1}, \cdots, w_{n}$ on $\tilde{U}$ such that $\check{v}_{j}=d_{j} {\rm Re} dw_{j}$, $j=1, \cdots, n$, which give  special holomorphic coordinates on $\tilde{U}\backslash q^{-1}(E)$.

Given any point in $\tilde{U}\backslash q^{-1}(E)$, we can define the conjugate coordinates $w_{1}^{*}, \cdots, w_{n}^{*}$ in a neighborhood of the point by ${\rm Re} dw_{j}^{*}=\check{v}_{j+n}$. This way we obtain that the holomorphic multivalued matrix  function $Z=[Z_{ij}]$ on $\tilde{U}\backslash q^{-1}(E)$ is in fact equal to $Z_{ij}=\frac{\partial w_{j}^{*}}{\partial w_{i}}$.

Lemma \ref{le02} implies   that there are $b^{p}_{ij}\in\mathbb{Q}$, $p=1, \cdots, k$, $1\leq i,j\leq n$, such that $$Z_{ij}(t)=Q_{ij}(t)+\sum_{p=1}^{k}\frac{\log t_{p}}{2\pi \sqrt{-1}}b^{p}_{ij}.$$
If we denote  $G_{ij}=\sum\limits_{p}Z_{ip}\frac{\partial w_{p}}{\partial t_{j}}=\frac{\partial w_{i}^{*}}{\partial t_{j}}$,  then $dw^{*}_{i}=\sum\limits_{j}G_{ij}dt_{j}$, and $$G_{ij}(t)=A_{ij}(t)+\sum_{p=1}^{k}\frac{\log t_{p}}{2\pi \sqrt{-1}}B^{p}_{ij}(t),$$ where $A_{ij}(t)$ and $B^{p}_{ij}(t)$ are holomorphic functions on $\tilde{U}$.
 By $0=ddw^{*}_{i}=d\left(\sum\limits_{j}G_{ij}dt_{j}\right)$, we have
 $$\frac{\partial G_{ij}}{\partial t_{\ell}}=\frac{\partial G_{i\ell}}{\partial t_{j}},$$ for any $j \neq \ell$, and
  $$ \frac{\partial A_{ij}}{\partial t_{\ell}}+\frac{\varepsilon (\ell) B^{\ell}_{ij}}{2\pi \sqrt{-1} t_{\ell}}+\sum_{p=1}^{k}\frac{\log t_{p}}{2\pi \sqrt{-1}}\frac{\partial B^{p}_{ij}}{\partial t_{\ell}}=\frac{\partial A_{i\ell}}{\partial t_{j}}+\frac{\varepsilon (j)B^{j}_{i\ell}}{2\pi \sqrt{-1} t_{j}}+\sum_{p=1}^{k}\frac{\log t_{p}}{2\pi \sqrt{-1}}\frac{\partial B^{p}_{i\ell}}{\partial t_{j}},  $$ where $\varepsilon(x)=1$ if $1\leq x \leq k$, and $\varepsilon(x)=0$ if $k+1\leq x \leq n$.
  For any given $1\leq \ell\leq k$, and $j\neq \ell$ and any $i$, we use this equation near a point where $t_\ell=0$ and $t_j\neq 0$, and see that $B^\ell_{ij}/t_\ell$ blows up at worst logarithmically near $t_\ell=0$.
  Therefore  $B^{\ell}_{ij}=t_{\ell}\tilde{B}^{\ell}_{ij}$ and  $B^{j}_{i\ell}=t_{j}\tilde{B}^{j}_{i\ell}$,  for any $j \neq \ell\leq k$, where $\tilde{B}^{\ell}_{ij}$ are holomorphic functions.  We obtain that
  \begin{equation}\label{log}
  G_{ij}(t)=A_{ij}(t)+\frac{\ve(j)\log t_{j}}{2\pi \sqrt{-1}}B^{j}_{ij}(t)+\sum_{p\neq j}\frac{t_p\log t_{p}}{2\pi \sqrt{-1}}\ti{B}^{p}_{ij}(t),
  \end{equation}
 where the summation is for $1\leq p\leq k$ and $p\neq j$. Since $|t\log t|\to 0$ when $t$ approaches  $0$ (for any given branch of $\log$), we get
  $$G_{ij}(t)=\frac{\varepsilon(j)\log t_{j}}{2\pi \sqrt{-1}}B^{j}_{ij}+O(1).$$    Since
$$q^{*}\omega=\frac{1}{4}\sum_{ij} (Z_{ij}-\overline{Z}_{ij}) dw_{i}\wedge d\bar{w}_{j}=\frac{1}{4}\sum_{i}(dw_{i}^{*}\wedge d\bar{w}_{i}-dw_{i}\wedge d\bar{w}^{*}_{i}), $$ we have
\begin{equation}\label{lab}
g_{i\bar{j}}=\frac{-\sqrt{-1}}{4}\sum_{p}\left(\frac{\partial w_{p}^{*}}{\partial t_{i}}\overline{\frac{\partial w_{p}}{\partial t_{j}}}-\frac{\partial w_{p}}{\partial t_{i}}\overline{\frac{\partial w_{p}^{*}}{\partial t_{j}}}\right)=\frac{-\sqrt{-1}}{4}\sum_{p}\left(G_{pi}\overline{\frac{\partial w_{p}}{\partial t_{j}}}-\frac{\partial w_{p}}{\partial t_{i}}\overline{G_{pj}}\right).
\end{equation}
We obtain the conclusion that $$ |g_{i\bar{j}}|\leq C(1-\varepsilon(i)\log |t_{i}|-\varepsilon(j)\log |t_{j}|), \  \ \ i,j=1, \cdots, n,$$ for a uniform constant $C>0$. Furthermore the
(single-valued) functions $g_{i\ov{j}}$ with $i,j>k$ extend continuously to $\ti{U}$ since in \eqref{lab} only the (multi-valued) functions $G_{pi}$ with $i>k$ are involved, so we can use \eqref{log} and again that $|t\log t|\to 0$ as $t\to 0$, for any given branch of $\log$.
\end{proof}

\begin{remark}
In the case when $f:M\to N$ is an elliptic surface, then Kodaira's classification of the possible singular fibers \cite{Ko} gives explicit formulas for the (multivalued) period function $\tau(y)$, such that $\check{\Lambda}_y=\mathrm{Span}_{\mathbb{Z}}(1,\tau(y))$, from which one can explicitly see that, after a branched covering $y=t^k$, the only possible singularities of $\tau(t)$ are of the form $\log t$ (see e.g. \cite[p.377]{He}). Lemma \ref{le02} is a (well-known) higher-dimensional generalization of this observation. Very similar (and more precise) estimates were obtained by Hwang-Oguiso \cite{HO} under more restrictive assumptions.
\end{remark}

\section{Applications to  SYZ for hyperk\"ahler manifolds}\label{sectsyz}
In this section we  apply   Theorem \ref{theorem1} to a refined SYZ conjecture due to  Gross-Wilson (\cite[Conjecture 6.2]{GW}), Kontsevich-Soibelman (\cite[Conjectures 1 and 2]{KS}) and Todorov (\cite[p. 66]{Man}) (see also \cite{Fuk}).

\subsection{Metric SYZ} Let   $X$ be   a Calabi-Yau $n$-manifold, and  $\mathfrak{M}_{X}$ be the moduli space of complex deformations of  $X$. If $\overline{\mathfrak{M}}_{X}$ denotes a certain  compactification, then  a large complex limit point  $p\in \overline{\mathfrak{M}}_{X}$ is a point representing the `worst possible degeneration' of the complex structures, which can be formulated via Hodge  theory (cf. \cite{Morr}).  Mirror symmetry predicts that for any large complex limit point $p\in \overline{\mathfrak{M}}_{X}$, there is an  another Calabi-Yau manifold $\check{X}$, called the mirror,  and an isomorphism between a neighborhood of $p$ in $\overline{\mathfrak{M}}_{X}$ and a neighborhood of a large radius limit in the complexified K\"{a}hler moduli space of $\check{X}$, which preserves some additional structures such as Yukawa couplings. Here a large radius limit point means the limit of $\exp 2\pi \sqrt{-1}(\mathbf{B}+s\sqrt{-1}\check{\omega})$, when $s\rightarrow \infty$,  in a certain compactification of $H^{2}(\check{X}, U(1))+\sqrt{-1}\mathbb{K}_{\check{X}}$, where $\mathbb{K}_{\check{X}}$ is the K\"{a}hler cone, $\check{\omega}$ is a K\"{a}hler metric,  and $\mathbf{B}\in H^{2}(\check{X}, U(1))$ is called a B-field.

 In  \cite{SYZ}, Strominger, Yau and Zaslow proposed  a  conjecture, the so called SYZ conjecture,  for constructing mirror Calabi-Yau manifolds via dual special Lagrangian fibrations. More precisely,
    the SYZ conjecture says  that near a large complex structure limit point $p$, the corresponding Calabi-Yau manifolds should admit a special Lagrangian torus fibration  such that the mirror $\check{X}$ should be obtained as a compactification of the dual torus fibration, after suitable instanton corrections induced from  the singular fibers. This has generated an immense amount of work, and we refer the reader to the surveys \cite{ABC,gross,GHJ} and references therein for more information.

 Later, a metric version of the SYZ conjecture was proposed by  Gross, Wilson, Kontsevich,  Soibelman and Todorov \cite{GW,KS,KS2,Man} by using   the collapsing of Ricci-flat K\"{a}hler metrics, which is also related to non-Archimedean geometry (cf. \cite{BJ}). Let $X_{t}$, $t\in (0,1]$, be a family of  $n$-dimensional Calabi-Yau manifolds such that the complex structures of $X_{t}$ converge  to a large complex limit point $p$ in $\overline{\mathfrak{M}}_{X}$.  The metric  SYZ conjecture \cite[Conjecture 6.2]{GW}, \cite[Conjecture 1]{KS}, asserts that there are  Ricci-flat K\"ahler metrics  $\ti{\omega}_{t}$ on $X_{t}$,   for $t\neq 0$, such that
       $(X_{t}, {\rm diam}_{\ti{\omega}_{t}}^{-2}(X_{t}) \ti{\omega}_{t})$ converges to a compact metric space $(Y,d_{Y})$  in the Gromov-Hausdorff sense, when $t\rightarrow 0$.
 Furthermore, there is an open and dense subset $Y_0\subset Y$ which is a smooth real $n$-dimensional Riemannian manifold $(Y_0,g)$,  and   admits    a real  affine structure. The singular locus $S_{Y}=Y\backslash Y_{0}$ is of Hausdorff codimension at least $2$.  The metric space  $(Y, d_{Y})$ is the metric completion of $(Y_{0}, g) $, and $ g$ is  a Monge-Amp\`ere  metric  on $Y_{0}$, i.e.  in local  affine coordinates $(y_{1},  \cdots, y_{n})$,  there is a potential function $\phi$ such that $$g= \sum_{ij} \frac{ \partial^{2} \phi}{ \partial y_{i}  \partial y_{j}} dy_{i} dy_{j},    \   \  {\rm and}  \  \   \det \Big(\frac{\partial^{2}\phi}{ \partial y_{i}  \partial y_{j} }\Big ) =c,$$
 for some $c\in\mathbb{R}_{>0}$.   When $\ti{\omega}_t$ have holonomy $SU(n)$ (resp. hyperk\"ahler),
 $Y$
 should  be homeomorphic to  an $n$-sphere (resp.  $\mathbb{CP}^n$).
 It is not hard to see that the conjecture is true when $X_t$ are Abelian varieties (see e.g. \cite{Od}).
This conjecture was verified by Gross and Wilson for elliptically fibered  K3 surfaces with only type $I_{1}$ singular fibers  in \cite{GW}, for large complex structure limits which arise as hyperk\"ahler rotations from our setup in the Introduction.  In  \cite{GTZ2},   Gross-Wilson's result was extended  to all elliptically fibered K3 surfaces, and
a partial results for higher dimensional hyperk\"ahler manifolds were obtained in \cite{GTZ, GTZ2}. As mentioned in the Introduction, Corollary \ref{ksthm} proves this conjecture for all large complex structure limits of projective hyperk\"ahler manifolds which arise from our setup via hyperk\"ahler rotation.  An  analogue   of this
conjecture was proved for canonically polarized manifolds  in \cite{Zh2}.

The next step in the SYZ program is to  construct  the mirror $ \check{X}$ as a certain compactification of $TY_{0}/\Lambda$, for a lattice subbundle $\Lambda$ of $TY_{0}$.  This is the so called the reconstruction problem, and is of great interest in mirror symmetry
(see  \cite{gross,KS2,GS}). The reconstruction problem suggests a more explicit behaviour of the Ricci-flat K\"{a}hler metrics $\ti{\omega}_{t}$ near the collapsing limit \cite[Conjecture 2]{KS}, which asserts that  $\ti{\omega}_{t}$ is asymptotic to certain  semi-flat Ricci-flat K\"{a}hler metrics.

\subsection{Semi-flat  hyperk\"{a}hler structures}  In this subsection, we recall the construction of   semi-flat hyperk\"{a}hler structures on algebraic  completely   integrable systems \cite{Fr}, and the semi-flat  SYZ construction  studied in \cite{Hi,Hi1}. The next subsection applies Theorem \ref{theorem1} to the metric  version of SYZ conjecture  for compact  hyperk\"{a}hler manifolds, and shows that the hyperk\"{a}hler structures approach  such semi-flat hyperk\"{a}hler structures near the limit.   We use the same notations  as in subsection \ref{4.2}.

Let  $(f: M_{0} \rightarrow N_{0}, [\alpha], \Omega)$ be an  algebraic  completely   integrable system, $\omega$ be the induced  special K\"{a}hler metric on $N_{0}$, and $g$ be the corresponding Riemannian metric.  The fibers $M_y=f^{-1}(y),y\in N_0$, are Abelian varieties  of type $(d_1,\dots,d_n)$.
Assume furthermore that there is a holomorphic Lagrangian  section $\sigma:N_0\to M_0$, which as mentioned in subsection \ref{4.2} implies that $M_0\cong T^{*(1,0)}N_0/\check{\Lambda}$ for a lattice subbundle $\check{\Lambda}$, and as before let $p: T^{*(1,0)}N_{0} \rightarrow M_{0}$ be the holomorphic covering map, which satisfies  $\Omega_{can}=p^{*}\Omega$, and $\check{\Lambda}=\ker p$.

For any $ t\in (0,1]$, we define a family of semi-flat  K\"{a}hler metrics  on $M_{0}$ by \begin{equation}\label{eq6.7}\omega_{\mathrm{SF},t}=t^{-\frac{1}{2}}f^{*}\omega+t^{\frac{1}{2}}\omega_{\rm SF},\end{equation} where $\omega_{\rm SF}$ is given by Theorem \ref{semiflat}, which satisfies  $\omega_{\mathrm{SF},t}|_{M_{y}}=\sqrt{t}\omega_{\rm SF}|_{M_{y}}\in \sqrt{t}[\alpha_{y}]$,  for any $y\in N_{0}$, and we denote by $g_{\mathrm{SF},t}$ the corresponding Riemannian  metric.  Note that $t^{\frac{1}{2}}\omega_{\mathrm{SF},t}$ is the semi-flat K\"{a}hler metric constructed in Section 2 of \cite{GTZ}. For any fiber $M_{y}$,  the diameter  $${\rm diam}(M_{y}, g_{\mathrm{SF},t}|_{M_{y}})\leq Ct^{\frac{1}{4}} \rightarrow 0,  \  \ {\rm when} \  \   t\rightarrow 0,$$ and thus $(M_{0}, g_{\mathrm{SF},t})$ collapses the torus fibers.

For an open subset $U\subset N_{0}$, let   $y_{1}, \cdots, y_{2n}$ be   the flat  Darboux coordinates such that  $dy_{i}$ satisfy (\ref{eq4.2++}).
   For the local trivialization $ T^{*}N_{0}|_{U}\cong U\times \mathbb{R}^{2n}$ by $\sum\limits_{i} x_{i}dy_{i} \mapsto (y_{1}, \cdots, y_{2n}, x_{1}, \cdots, x_{2n})$,  $dx_{1}, \cdots, dx_{2n}$ are well-defined closed  1-forms on $f^{-1}(U)$, and we have   \begin{equation}\label{eq6.6}\omega_{\rm SF}=- \sum_{i=1}^{n}dx_{i}\wedge dx_{i+n}=\frac{\sqrt{-1}}{2}\sum\limits_{ij}{\rm Im}Z_{ij}^{-1}\vartheta_{i}\wedge \bar{\vartheta}_{j}, \end{equation} where $\vartheta_{i}=d x_{i}- \sum\limits_{j=1}^{n}Z_{ij}d x_{j+n}$, $i=1, \cdots, n$, which may not be closed (cf. \cite[Lemma 3.3]{He}).
  Thus  $$\omega_{\mathrm{SF},t}=t^{-\frac{1}{2}}\sum_{i=1}^{n}dy_{i}\wedge dy_{i+n}-t^{\frac{1}{2}}\sum_{i=1}^{n}dx_{i}\wedge dx_{i+n}.$$
In particular, we see that $p^{*}\omega_{\mathrm{SF},t}^{2n}= \Omega_{can}^{n}\wedge \overline{\Omega}_{can}^{n}$, which shows that,
by changing variables if necessary, $(p^{*}\omega_{\rm SF,t}, \Omega_{ can}=p^{*} \Omega)$ is the hyperk\"ahler structure on $T^{*(1,0)}N_{0}$ constructed in Section 2 of \cite{Fr} (see also \cite[Section 3.2]{He}), and
$(\omega_{\mathrm{SF},t}, \Omega, g_{\mathrm{SF},t})$, $ t\in (0,1]$,  is a family of hyperk\"{a}hler structures on $M_{0}$.

   By hyperk\"{a}hler rotation, we define a family of complex structures $J_{t}$ with hyperk\"{a}hler structures \begin{equation}\label{eq6.8}\omega_{J_{t}}={\rm Re}\Omega,  \  \  \ \Omega_{J_{t}}={\rm Im}\Omega+\sqrt{-1}\omega_{\mathrm{SF},t},\end{equation}  and the fibration $f:M_{0}\rightarrow N_{0}$ is a special Lagrangian fibration with respect to $(\omega_{J_{t}}, \Omega_{J_{t}}^{n})$.

  Note that $\omega= \sum\limits_{i}dy_{i}\wedge dy_{i+n}$ under the flat  Darboux coordinates $y_{1}, \cdots, y_{2n}$, and $g$ is a  Monge-Amp\`ere  metric with the local potential $\phi$, i.e.  $\phi_{ij}=\frac{\partial^{2} \phi}{\partial y_{i} \partial y_{j}}=g_{ij}$ as shown in subsection \ref{4.1}.   The Legendre transform of the local potential function  $\phi$ gives the dual affine structure, which  is defined  by the local  dual affine  coordinates $\phi_{1}=\frac{\partial \phi}{\partial y_{1} }, \cdots, \phi_{2n}=\frac{\partial \phi}{\partial y_{2n} }$.

 \begin{lemma}\label{lem6.0} The K\"{a}hler form  $\omega_{J_{t}}$ is induced by the canonical symplectic form on $T^{*}N_{0}$, i.e.  $$\omega_{J_{t}}  =\sum_{i=1}^{2n}dy_{i}\wedge dx_{i}.  $$ If we let   $$\chi_{t,i}=\exp 2\pi\sqrt{-1}(x_{i}+\sqrt{-1}t^{-\frac{1}{2}}\phi_{i}), \  \  \  i=1, \cdots, 2n, $$   then $\chi_{t,1}, \cdots, \chi_{t,2n}$ are holomorphic Darboux coordinates on $f^{-1}(U)$  with respect to $J_{t}$, and
$$ \Omega_{J_{t}}=\frac{t^{\frac{1}{2}}}{4 \pi^{2}\sqrt{-1}}\sum_{i=1}^{n} \frac{d \chi_{t,i}}{ \chi_{t,i}}\wedge \frac{d \chi_{t,i+n}}{ \chi_{t,i+n}}.$$
\end{lemma}

\begin{proof} Denote by $I$ the complex structure on $N_{0}$.
By $g(\cdot,\cdot)=\omega ( \cdot, I \cdot)$, if    $I(\frac{\partial}{\partial y_{i}})=\sum\limits_{j}\frac{\partial}{\partial y_{j}}I_{ji}$, then  $$\begin{bmatrix}
     0, - {\rm id} \\ {\rm id}, \    0
     \end{bmatrix} [\phi_{ij}] =[I_{ij}],$$  i.e.
     $  I_{ij}=-\phi_{i+n,j},  \ \ I_{i,j+n}=-\phi_{i+n,j+n},  \ \ I_{i+n,j}=\phi_{i,j},  I_{i+n,j+n}=\phi_{i,j+n},$     for $1\leq i\leq n$. We have $$I(dy_{i})=\sum\limits_{j}I_{ij}dy_{j}=-d\phi_{i+n}, \  \  \   I(dy_{i+n})=\sum\limits_{j}I_{i+n,j}dy_{j}=d\phi_{i},$$
     $$dw_{i}=d(y_{i}+\sqrt{-1} \phi_{i+n}),\  \  \   dw_{i}^{*}=-d(y_{i+n}-\sqrt{-1} \phi_{i}),$$ for $i=1, \cdots, n$, where the special coordinates $w_1,\dots,w_n$ and their conjugates $w_1^*,\dots,w_n^*$ are defined as in subsection \ref{4.2}.
     By $\omega(\cdot,\cdot)=\omega(I\cdot,I\cdot)$,  $$\omega= \sum\limits_{i}dy_{i}\wedge dy_{i+n}=\sum\limits_{i}I(dy_{i})\wedge I(dy_{i+n})=\sum\limits_{i}d\phi_{i}\wedge d\phi_{i+n},\ \  \ {\rm and}$$
    $$\omega_{\mathrm{SF},t}= \sum_{i}(t^{-\frac{1}{2}}d\phi_{i}\wedge d\phi_{i+n}-t^{\frac{1}{2}} dx_{i}\wedge dx_{i+n}). $$

Under  the  local trivialization $ T^{*(1,0)}N_{0}|_{U}\cong U\times \mathbb{C}^{n}$ by $\sum\limits_{i} z_{i}dw_{i} \mapsto \\ (w_{1}, \cdots, w_{n}, z_{1}, \cdots, z_{n})$, we have  $$\Omega_{can}=\sum\limits_{i} d w_{i}\wedge dz_{i}, \  \  \  {\rm and  }  \  \   z_{i}=x_{i}- \sum\limits_{j=1}^{n}Z_{ij}x_{j+n}$$  by  $\left(\sum\limits_{i=1}^{2n} x_{i}dy_{i}\right)^{(1,0)}=\sum\limits_{i=1}^{n} \left(x_{i}-\sum\limits_{j=1}^{n}Z_{ij}x_{j+n}\right)dw_{i}, \  \  1\leq i\leq n.$
      Then   \[\begin{split}\Omega & =-d(\sum_{i}z_{i}dw_{i}) =
-d(\sum_{i}(x_{i}dw_{i} -x_{i+n}dw_{i}^{*})) \\ & =\sum_{i=1}^{2n}dy_{i}\wedge dx_{i} +\sqrt{-1}\sum_{i=1}^{n}(d\phi_{i+n}\wedge dx_{i}-d\phi_{i}\wedge dx_{i+n}). \end{split}\]

Thus $$\omega_{J_{t}}  =\sum_{i=1}^{2n}dy_{i}\wedge dx_{i}, \  \  \  {\rm and}$$
$$ \Omega_{J_{t}}= \sum_{i=1}^{n} (d\phi_{i+n}\wedge dx_{i}-d\phi_{i}\wedge dx_{i+n}+\sqrt{-1}(t^{-\frac{1}{2}}d\phi_{i}\wedge d\phi_{i+n}-t^{\frac{1}{2}} dx_{i}\wedge dx_{i+n})).  $$ We obtain the conclusion by using $\frac{d \chi_{t,i}}{ \chi_{t,i}}=2\pi \sqrt{-1} (dx_{i}+t^{-\frac{1}{2}}\sqrt{-1}d \phi_{i}). $
\end{proof}

 We remark that on $f^{-1}(U)$, $f$ is the logarithmic map $2\pi \phi_{i}=-t^{\frac{1}{2}}\log |\chi_{t,i}|$, $i=1, \cdots, 2n$, with respect to the dual affine structure, which converts algebro geometric objects in $f^{-1}(U)$ into tropical geometric objects on $U$ when $t\rightarrow 0$.

Now we recall the  semi-flat SYZ construction of $(M_{0}, \omega_{J_{t}}, \Omega_{J_{t}})$ (cf.  \cite{Hi,Hi1}). We ignore B-fields in the following discussion.  Note that $J_{t}\frac{\partial}{\partial x_{i}}=t^{\frac{1}{2}}\frac{\partial}{\partial \phi_{i}}$ and
        $$\omega_{J_{t}}  =\sum_{i=1}^{2n}dy_{i}\wedge dx_{i}  =\sum_{ij}\frac{\partial y_{i}}{\partial \phi_{j}}d\phi_{j}\wedge dx_{i}=\sum_{ij} \phi_{ij}^{-1}d\phi_{j}\wedge dx_{i}. $$
  Thus   for any   $y\in N_{0}$,  we obtain that    $$ g_{\mathrm{SF},t}|_{M_{y}}=t^{\frac{1}{2}} \sum_{ij}\phi_{ij}^{-1}(y)dx_{i}dx_{j}.  $$

Following \cite{SYZ}, we now construct the semi-flat SYZ mirror of $(M_0,\omega_{J_t},J_t)$, as follows. Let   $\check{M}_{0}=TN_{0}/\Lambda$, and $\check{f}: \check{M}_{0} \rightarrow N_{0}$ be the fibration induced by $TN_{0} \rightarrow N_{0}$.
  For any $y\in N_{0}$, the fiber $\check{M}_{y}=T_{y}N_{0}/\Lambda_{y}$ is  the dual Abelian variety of $M_{y}$, which is of type $(d_{n}/d_{n}, \cdots, d_{1}/d_{n})$.
On $TN_{0}$, there is a natural complex structure induced by the flat affine structure on $N_{0}$, which gives a complex structure $\check{J} $ on $\check{M}_{0}$.
 Under a  local trivialization $ TN_{0}|_{U}\cong U\times \mathbb{R}^{2n}$ by $\sum\limits_{i} \check{x}_{i}\frac{\partial}{\partial y_{i}} \mapsto (y_{1}, \cdots, y_{2n}, \check{x}_{1}, \cdots, \check{x}_{2n})$, the complex structure $\check{J} $ is given by the holomorphic coordinates $\xi_{i}=\exp 2 \pi \sqrt{-1 }(\check{x}_{i}+\sqrt{-1}y_{i})$,  $  i=1, \cdots, 2n$.  Note that if  $y_{i}'$, $i=1, \cdots, 2n$, are another flat  Darboux coordinates, and $\check{x}_{i}'$ are induced coordinates with  $\sum\limits_{i} \check{x}_{i}\frac{\partial}{\partial y_{i}}=\sum\limits_{i} \check{x}_{i}'\frac{\partial}{\partial y_{i}'} $, then  $y_{i}=\sum\limits_{j}a_{ij}y_{j}'+b_{i}$ and $\check{x}_{i}=\sum\limits_{j}a_{ij}\check{x}_{j}'$, where $(a_{ij})\in Sp(2n, \mathbb{R})$ and $b_{i}\in \mathbb{R}$.  Therefore $$\check{\Omega}=\frac{-1}{4\pi^{2}}\sum_{i=1}^{n}\frac{d \xi_{i}}{\xi_{i}}\wedge \frac{d \xi_{i+n}}{\xi_{i+n}}$$ is a well-defined holomorphic symplectic form on $\check{M}_{0}$.

  A   natural K\"{a}hler metric on $\check{M}_{0}$ is    \begin{equation}\label{eq6.80}\check{\omega}_{t}=t^{-\frac{1}{2}}\sum_{i=1}^{2n}d \phi_{i}\wedge d \check{x}_{i}=t^{-\frac{1}{2}}\sum_{ij}\frac{\partial^{2} \phi}{\partial y_{i}\partial y_{j}}d y_{i} \wedge d \check{x}_{j}=t^{-\frac{1}{2}}\check{\omega}, \end{equation} which gives a hyperk\"{a}hler structure $(\check{\omega}_{t}, t^{-\frac{1}{2}}\check{\Omega})$  on $\check{M}_{0}$ since $ \det (\phi_{ij}) \equiv {\rm const}$.
If $\check{g}_{\mathrm{SF},t}$ denotes  the Riemannian metric determined by $\check{\omega}_{t}$ and $\check{J}$, then
    $$ \check{g}_{\mathrm{SF},t}|_{\check{M}_{y}}=t^{-\frac{1}{2}} \sum_{ij}\phi_{ij}(y)d \check{x}_{i}d \check{x}_{j}, $$
  by $\check{J}\frac{\partial}{\partial \check{x}_{i}}=\frac{\partial}{\partial y_{i}}$, for a  $y\in N_{0}$.
The semi-flat SYZ mirror of $(M_{0}, \omega_{J_{t}}, J_{t})$ is  $(\check{M}_{0}, \check{\omega}_{t}, \check{J})$ in the sense of T-duality (cf. \cite{SYZ} and Chapter 1.3 in \cite{ABC}), i.e. $(\check{M}_{y}, \check{g}_{\mathrm{SF},t}|_{\check{M}_{y}})$ is the dual torus of $(M_{y}, g_{\mathrm{SF},t}|_{M_{y}})$ for any $y\in N_{0}$.   When  $t\rightarrow 0$, we say that the complex structures $J_{t}$ tends to a {\rm large complex limit}, while the symplectic structure $\omega_{J_{t}}$ is fixed, in the sense that its semi-flat SYZ mirror has the symplectic structures $\check{\omega}_{t}=t^{-\frac{1}{2}}\check{\omega}$ tending to a {\rm large radius limit} while keeping the complex structure $\check{J}$ fixed.

\subsection{Collapsing hyperk\"ahler metrics are close to semi-flat}
Now  we show  how   Theorem \ref{theorem1}, together with \cite{GTZ,GTZ2},   fits  into this refined version of SYZ conjecture  for hyperk\"{a}hler manifolds. The setup is now the same as in Theorem \ref{theorem1},  so $f:M^{2n}\to N^n\cong\mathbb{CP}^n$ is a holomorphic fiber space with $M$ projective hyperk\"ahler, with the extra assumption that there is a holomorphic Lagrangian  section $\sigma:N_0\to M_0$. We denote by $\Omega$ the holomorphic symplectic form on $M$, the fibration $f$ is then an algebraic completely integrable system over $N_0$, the complement of the discriminant locus of $f$ in $N$, and $[\alpha]$ is an integral K\"ahler class on $M$. The fibers $M_y=f^{-1}(y),y\in N_0$, are Abelian varieties, and the polarization $[\alpha_{y}]=[\alpha|_{M_{y}}]$ is of type $(d_1,\dots,d_n)$. As we mentioned in subsection \ref{4.2}, the existence of $\sigma$ implies that $M_0\cong T^{*(1,0)}N_0/\check{\Lambda}$, and let then $p: T^{*(1,0)}N_{0} \rightarrow M_{0}$ be the holomorphic covering map, which satisfies  $\Omega_{can}=p^{*}\Omega$, and $\check{\Lambda}=\ker p$.

 Let $[\alpha_{0}]$ be the  ample class on $N$ such that $$[\alpha]^{n}\cdot [f^{*}\alpha_{0}]^{n}= \int_{M}\Omega^{n}\wedge \bar{\Omega}^{n},  $$
  and $\tilde{\omega}_{t}$ be the unique Ricci-flat hyperk\"ahler metric on $M$ in the class $f^{*}[\alpha_{0}]+t[\alpha], 0<t\leq 1,$ which satisfies the complex Monge-Amp\`ere equation $$\tilde{\omega}_{t}^{2n}=c_{t}t^{n}\Omega^{n}\wedge \bar{\Omega}^{n} $$ with $c_{t} \rightarrow 1$ when $t\rightarrow 0$.
   Therefore, $(c_{t}^{-\frac{1}{2n}}t^{-\frac{1}{2}}\tilde{\omega}_{t}, \Omega)$ is a hyperk\"{a}hler structure, and
   we denote $\hat{g}_{t}$ the corresponding  hyperk\"{a}hler metric of $c_{t}^{-\frac{1}{2n}}t^{-\frac{1}{2}}\tilde{\omega}_{t}$.   By hyperk\"{a}hler rotation,  we have  a family of complex structures $\tilde{J}_{t}$ with hyperk\"{a}hler structures $$\omega_{\tilde{J}_{t}}={\rm Re}\Omega,  \  \  \ \Omega_{\tilde{J}_{t}}={\rm Im}\Omega +\sqrt{-1}c_{t}^{-\frac{1}{2n}}t^{-\frac{1}{2}}\tilde{\omega}_{t}.$$
    A well-known simple calculation shows that the fibration $f:M\rightarrow N$ becomes a  special Lagrangian fibration  with respect to $(\omega_{\tilde{J}_{t}}, \Omega_{\tilde{J}_{t}}^{n})$, and $\sigma$ becomes  a special Lagrangian section.

By \cite[Theorem 1.2]{GTZ} and $c_{t} \rightarrow 1$,  we have that by passing to  subsequences,  $(M, t^{\frac{1}{2}}\hat{g}_{t}, \tilde{\omega}_{t})$  converges to a compact metric space $(X, d_{X})$ in the Gromov-Hausdorff sense, and there is a locally isometric embedding $(N_{0}, \omega)\hookrightarrow (X, d_{X})$, and  \cite[Theorem 1.2]{GTZ2} asserts that $\omega$ is a special K\"{a}hler metric on $N_{0}$. Furthermore, Lemma 4.1 in  \cite{GTZ2} shows that
 $\omega $ is  the special K\"{a}hler metric induced by the algebraic   completely   integrable system $(f: M_{0} \rightarrow N_{0}, [\alpha], \Omega)$, where $M_{0}=f^{-1}(N_{0})$.
Now Theorem \ref{theorem1}  shows that $(X, d_{X})$ is the metric completion of $(N_{0}, \omega)$, it is homeomorphic to $\mathbb{CP}^n$, its singular set $X\backslash N_{0}$ has Hausdorff codimension at least $2$, and there is no need to pass to any subsequence in   the convergence, i.e.  $$(M,  t^{\frac{1}{2}}\hat{g}_{t}, \tilde{\omega}_{t})\rightarrow (X, d_{X}),
       \  \  {\rm when }\  \ t \rightarrow 0. $$  Furthermore, as predicted by \cite[Conjecture 2]{KS}, we claim that $t^{\frac{1}{2}}\hat{g}_{t}$ approaches some semi-flat metrics in a certain sense that we now explain.

As in \eqref{eq6.7}, for any $ t\in (0,1]$ we define a family of semi-flat  K\"{a}hler metrics  on $M_{0}=f^{-1}(N_0)$ by \begin{equation}\label{eq6}\omega_{\mathrm{SF},t}=t^{-\frac{1}{2}}f^{*}\omega+t^{\frac{1}{2}}\omega_{\rm SF},\end{equation}
where $\omega_{\rm SF}$ is given by Theorem \ref{semiflat}, and we denote by $g_{\mathrm{SF},t}$ the corresponding Riemannian  metric.
Following \cite{GTZ}, we define the dilation map  $\lambda_{t}: T^{*(1,0)}N_{0} \to T^{*(1,0)}N_{0}$ by $\lambda_{t}(y,z)=(y, t^{-\frac{1}{2}}z)$ and  the  covering map $p: T^{*(1,0)}N_{0} \rightarrow M_{0}$ such that  $\Omega_{can}=p^{*}\Omega$, as in subsection \ref{4.2}.  Thanks to Proposition \ref{conver} we have that $$\lambda_{t}^{*}p^{*} \tilde{\omega}_{t} \to p^{*}\omega_{\mathrm{SF},1},$$ smoothly on compact sets. Direct calculations show that  $$\lambda_{t}^{*}p^{*} \sqrt{t}\Omega = \Omega_{ can}, \  \  {\rm and} \ \ \lambda_{t}^{*}p^{*}t^{\frac{1}{2}} \omega_{\mathrm{SF},t} = \omega_{\mathrm{SF},1} $$  (cf. Section 4 in \cite{GTZ2}), which implies $$\|c_{t}^{-\frac{1}{2n}}\tilde{\omega}_{t} -t^{\frac{1}{2}}\omega_{\mathrm{SF},t}\|_{C^{\infty}_{\rm loc}(M_0,t^{\frac{1}{2}}g_{\mathrm{SF},t})}\rightarrow 0.$$

 Note that $(\omega_{\mathrm{SF},t}, \Omega, g_{\mathrm{SF},t})$, $ t\in (0,1]$,  is a family of semi-flat  hyperk\"{a}hler structures on $M_{0}$.
By hyperk\"{a}hler rotation, we define a family of complex structures $J_{t}$ with hyperk\"{a}hler structures $$\omega_{J_{t}}={\rm Re}\Omega=\omega_{\tilde{J}_{t}},  \  \  \ \Omega_{J_{t}}={\rm Im}\Omega+\sqrt{-1}\omega_{\mathrm{SF},t}.$$ By Lemma \ref{lem6.0}, the K\"{a}hler form  $\omega_{J_{t}}$ is induced by the canonical symplectic form on $T^{*}N_{0}$.

 From this discussion together with Theorem \ref{theorem1}, we obtain the following theorem:

\begin{theorem}\label{syzthm}In the above setup we have \begin{itemize}
    \item[i)]On $M_{0}$, when $t\rightarrow 0$, we have $\omega_{J_{t}}=\omega_{\tilde{J}_{t}}$,
     $$\|t^{\frac{1}{2}}(\hat{g}_{t}-g_{\mathrm{SF},t})\|_{C^{\infty}_{\rm loc}(M_0,t^{\frac{1}{2}}g_{\mathrm{SF},t})}\rightarrow 0, \  \  {\rm and} \ \ \|\tilde{J}_{t}-J_{t}\|_{C^{\infty}_{\rm loc}(M_0,t^{\frac{1}{2}}g_{\mathrm{SF},t})}\rightarrow 0.$$
      \item[ii)]     There is a special K\"{a}hler metric $\omega$ on $N_{0}$ such that the metric completion $\overline{(N_{0}, g)}$ is compact, and   $$(M, t^{\frac{1}{2}}\hat{g}_{t})\rightarrow \overline{(N_{0}, g)}, $$ in the Gromov-Hausdorff sense, where $ g$ denotes the corresponding Riemannian metric of $\omega$ on $N_{0}$.
       \item[iii)] The singular set $S_{N_{0}}=\overline{(N_{0}, g)}\backslash N_{0}$ has the  Hausdorff  codimension at least $2$.
       \item[iv)] The space $\overline{(N_{0}, g)}$ is homeomorphic to $\mathbb{CP}^n$.
         \item[v)]  $g$  is  a real  Monge-Amp\`ere  metric with respect to the real affine structure  determined by the special  K\"{a}hler metric $\omega$.
\end{itemize}
\end{theorem}

Note that the semi-flat symplectic form $\omega_{J_{t}}$ on $M_{0}$ can be extended to a symplectic form on $M$, which equals $\omega_{\tilde{J}_{t}}$. However, the complex structure $J_{t}$ usually cannot be extended to a complex structure on $M$.  In order to extend $J_{t}$, one must add to it certain additional terms called  instanton corrections (these equal  $\tilde{J}_{t}-J_{t}$ in the present case), which  are determined by certain tropical geometric  objects on $N_{0}$  constructed inductively from the initial information of the singularities $S_{N_{0}}$. See \cite{Fuk} in the analytic setting, \cite{KS2,GS} in the algebro geometric setting, and \cite{GMN,N,KS3} for the current case of hyperk\"{a}hler manifolds.

Let us also remark that,   as shown in the last subsection,    the complex structures $J_{t}$ (therefore also $\tilde{J}_{t}$) tends to a  large complex limit, when  $t\rightarrow 0$, in the sense that its semi-flat SYZ mirror $\check{M}_{0}$  has the symplectic structures $\check{\omega}_{t}=t^{-\frac{1}{2}}\check{\omega}$ tending to a  large radius limit.    Furthermore,   we expect that $\check{\omega}_{t}$ extends to a symplectic form on  a certain compactification of $\check{M}_{0}$, if the SYZ mirrors of $M$ indeed exist, for example the case of Section 2 in \cite{GTZ}.

Lastly, let us mention that the conjecture of Gross-Wilson, Kontsevich-Soibelman and Todorov has directly inspired a purely algebro-geometric conjecture, which is as follows: let $\mathfrak{X}\to C$ be a projective family of Calabi-Yau manifolds over a quasiprojective curve, smooth over $C\backslash \{o\}$, such that $X_o$ is a large complex structure limit. After applying semistable reduction and a relative MMP, the dual intersection complex of the new central fiber is denoted by $Sk(\mathfrak{X})$, the essential skeleton of $\mathfrak{X}$. It is a connected $n$-dimensional simplicial complex, whose topological type does not depend on the choices we made. The conjecture is then that $Sk(\mathfrak{X})$ should topologically be an $n$-sphere when $\ti{\omega}_t$ (the Ricci-flat K\"ahler metric on $X_t$ in the polarization class) have holonomy $SU(n)$, and topologically $\mathbb{CP}^n$ when $\ti{\omega}_t$ are hyperk\"ahler. In particular, it should be homeomorphic to the Gromov-Hausdorff limit $(X,d_X)$ of the collapsing Ricci-flat K\"ahler metrics (normalized to have unit diameter). See \cite{BJ,MN,NX} for more details, and \cite{KLSV,KX} for very recent progress on these questions.

\end{document}